\documentclass[11pt]{amsart}

\usepackage{amsmath}
\usepackage{amssymb}
\usepackage{latexsym}
\usepackage{amscd}
\usepackage{citesort}
\usepackage{mathrsfs}
\usepackage[all]{xy}

\newdimen\AAdi%
\newbox\AAbo%
\def\AAk#1#2{\s_etbox\AAbo=\hbox{#2}\AAdi=\wd\AAbo\kern#1\AAdi{}}%
\def\AAr#1#2#3{\s_etbox\AAbo=\hbox{#2}\AAdi=\ht\AAbo\raise#1\AAdi\hbox{#3}}%
\font\tenmsb=msbm10 at 12pt \font\sevenmsb=msbm7 at 8pt
\font\fivemsb=msbm5 at 6pt
\newfam\msbfam
\textfont\msbfam=\tenmsb \scriptfont\msbfam=\sevenmsb
\scriptscriptfont\msbfam=\fivemsb

\newtheorem{theorem}{Theorem}

\newtheorem{remark}[theorem]{Remark}

\newtheorem{lemma}[theorem]{Lemma}
\newtheorem{proposition}[theorem]{Proposition}

\numberwithin{equation}{section} \numberwithin{theorem}{section}

\renewcommand{\topmargin}{0cm}
\renewcommand{\oddsidemargin}{5mm}
\renewcommand{\evensidemargin}{5mm}
\renewcommand{\textwidth}{150mm}
\renewcommand{\textheight}{230mm}

\def\R{\mathbb R}

\def\Sp{\mathbb S}

\def\na{\nabla}

\def\f#1#2{\frac{#1}{#2}}

\def\a{\alpha}
\def\be{\beta}

\def\r{\Re_{I\!V}}

\def\p#1{\partial #1}

\def\de{\delta}
\def\De{\Delta}
\def\e{\eta}
\def\ep{\epsilon}

\def\G{\Gamma}
\def\g{\gamma}

\def\la{\lambda}

\def\lan{\langle}
\def\ran{\rangle}

\def\Om{\Omega}
\def\th{\theta}

\def\si{\sigma}
\def\Si{\Sigma}

\def\r{\rho}

\def\div{\mathrm{div}}

\begin{document}

\title
{Entire spacelike translating solitons in Minkowski space}
\author{Qi Ding}
\address{Institute of Mathematics, Fudan University,
Shanghai 200433, China}
\email{09110180013@fudan.edu.cn}

\begin{abstract}
In this paper, we study entire spacelike translating solitons in Minkowski space. By constructing convex spacelike solutions to \eqref{tsl} in bounded convex domains, we obtain many entire smooth convex strictly spacelike translating solitons by prescribing boundary data at infinity.
\end{abstract}

\maketitle

\section{Introduction}

Let $\R^{n+1}_1$ be the Minkowski space $(\R^{n+1},\bar{g})$ with the Lorentz metric
$$\bar{g}=\sum_{i=1}^ndx_i^2-dx_{n+1}^2.$$
We will say that a hypersurface $\Si=\{(x,w(x))\mid x\in\Om\}\subset \R^{n+1}_1$ is \emph{strictly spacelike}, if $w\in C^1(\Om)$ and $|Dw|<1$ in $\Om$; $\Si$ is \emph{weakly spacelike}, if $w\in C^{0,1}(\Om)$ and $|Dw|\le1$ a.e. in $\Om$. Here, $C^{0,1}(\Om)$ is the class of locally Lipschitz functions in $\Om$. For convenience, we often call strictly (weakly) spacelike hypersurfaces by the functions whose graphs are them.

Mean curvature flow in Minkowski space is a family of smooth strictly spacelike embeddings
$X_t=X(\cdot,t): \R^n\rightarrow \R^{n+1}_1$
with corresponding images $M_t=X_t(\R^n)$ satisfying the following evolution equation
\begin{equation}\label{MCFE}
\f{dX}{dt}=\overrightarrow{H}
\end{equation}
on some time interval, where $\overrightarrow{H}$ is the mean curvature vector of $M_t$ in $\R^{n+1}_1$. Every $M_t$ is the graph of a function $U(\cdot,t)$ with $|DU(\cdot,t)|<1$. Equation \eqref{MCFE} is equivalent up to diffeomorphisms in $\R^n$ to the equation
\begin{equation}\label{MCFt}
\f{\p U}{\p t}=\sqrt{1-|DU|^2}\div\left(\f{DU}{\sqrt{1-|DU|^2}}\right),
\end{equation}where '$\div$' is the divergence on $\R^n$.
There is an important class of solutions for \eqref{MCFt} in $\R^{n}$ which moves by vertical translation. This solution is called \emph{Translating Soliton}, i.e.,  $$x_{n+1}=U(x,t)=u(x)+t,$$
where $u$ satisfies the elliptic equation
\begin{equation}\aligned\label{tsl}
\div\left(\f{Du}{\sqrt{1-|Du|^2}}\right)=\f1{\sqrt{1-|Du|^2}}\qquad \mathrm{in}\ \R^n.
\endaligned
\end{equation}

Mean curvature flow in the ambient Minkowski space and Lorentzian manifold has been studied extensively (see \cite{E1}\cite{E}\cite{E3}\cite{EH}\cite{HY} for example).
Translating solitons can be regarded as a natural way of foliating spacetimes by almost null like hypersurfaces. Particular examples may give insight into the structure of certain spacetimes at null infinity and have possible applications in general relativity \cite{E}. Convex translating solitons in Euclidean space $\R^{n+1}$, namely, the graphic functions satisfy
\begin{equation}\aligned\label{tseuc}
\div\left(\f{Du}{\sqrt{1+|Du|^2}}\right)=\f1{\sqrt{1+|Du|^2}}\qquad \mathrm{in}\ \R^n,
\endaligned
\end{equation}
are arising at type II singularities of mean curvature flow by Huisken and Sinestrari \cite{HS1}\cite{HS2}. X.-J. Wang in \cite{W1} showed that the convex solutions to \eqref{tseuc} must be rotationally symmetric in an appropriate coordinate system for $n=2$, and constructed non-rotationally symmetric entire convex translating solitons for $n\ge3$.

In the case $n=1$, $u_0(x)=\log \cosh x$ is a particular solution to \eqref{tsl}. In \cite{E}, Ecker constructed a radially symmetric solution to \eqref{tsl} for general $n$. Later, H.-Y. Jian \cite{J} gave a detailed discussion for this radially symmetric solution. In this paper, we only consider the case $n\ge2$ and construct many entire smooth convex strictly spacelike translating solutions which are asymptotic to all the functions in $Q$ but linear functions at infinity.

Here $Q$ is a set defined as follows: if $w\in Q$, then $w$ is a convex homogeneous of degree one function and $|Dw(y)|=1$ for any $y\in\R^n$ where the gradient exists. Each function $w\in Q$ could be represented as $w(x)=\sup_{\la\in\Lambda}\lan\la,x\ran$ where $\Lambda\subset\Sp^{n-1}$ is a closed set of unit vectors and $\lan\ ,\ \ran$ is the standard inner product in $\R^n$. Conversely, for every closed set $\Lambda\subset\Sp^{n-1}$, $\sup_{\la\in\Lambda}\lan\la,x\ran\in Q$ (please see \cite{CT}\cite{T} for details).
For any weakly spacelike function $w$ in $\R^n$, we define $$V_w(x)=\lim_{r\rightarrow\infty}\f{w(rx)}r$$ if such a limit exists for every $x\in\R^n$. And we say that $V_w$ is the \emph{blow down} of $w$. The blow down is well-defined for convex weakly spacelike functions \cite{T}.

A famous Calabi-Cheng-Yau \cite{C}\cite{CY} result tells us non-existence of nontrivial complete maximal spacelike hypersurface in Minkowski space. Cheng-Yau \cite{CY} also showed that any spacelike hypersurface $\Si$ of nonzero constant mean curvature has nonpositive Ricci curvature, namely that the function whose graph is $\Si$ is convex (see also \cite{T}). For every function $V\in Q$ but linear functions
Treibergs in \cite{T} could find a function $u$ blowing down to $V$ and the graph of $u$ is a strictly spacelike constant mean curvature hypersurface. Compared to this work of Treibergs we get the following Theorem for translating solitons.
\begin{theorem}\label{Vcu}
Let $Q$ be defined as above. For $n\ge2$ and any function $V$ in $Q$ except linear functions there is an entire smooth convex strictly spacelike solution $u$ to \eqref{tsl} such that $u$ blows down to $V$, namely,
$\lim_{r\rightarrow\infty}\f{u(rx)}r=V(x)$ for each $x\in\R^n.$
\end{theorem}

In \cite{A}, M. Aarons gave a conjecture which asks:
\emph{For $a>0$, whether there is a solution $u$ of
\begin{equation}\aligned\label{tsf}
\div\left(\f{Du}{\sqrt{1-|Du|^2}}\right)=\f a{\sqrt{1-|Du|^2}}+c\qquad \mathrm{in}\ \R^n,
\endaligned
\end{equation}
which blows down to every $V$ in $Q$.}

We may use the technique for proving Theorem \ref{Vcu} to study \eqref{tsf} with $c>0$, and obtain an existence result for \eqref{tsf} in $\S$ 6.

Every entire spacelike graph with constant mean curvature hypersurface in $\R^{n+1}_1$ is complete \cite{CY}.
However, translating solitons in $\R^{n+1}_1$ may be not complete, and the mean curvatures must be unbounded.
A nature question is whether every entire strictly spacelike function $u$ to \eqref{tsl} is convex? H.-Y. Jian also asked this question in \cite{J}.

In the present paper, we find a variational functional $F$ defined by \eqref{Var} for equation \eqref{tsl}, and show that any translating soliton is maximal for $F$, namely, any variation with compact support do not increase its area under the weight $e^{-x_{n+1}}$(see Theorem \ref{uni}). This help us to establish a comparison principle of weak solutions to $F$.
By constructing barrier functions we study light ray within the weak solution hypersurface, which plays a key role for showing that the limit function of a sequence of strictly spacelike solutions to \eqref{tsl} is strictly spacelike.

Through a calculation for the second fundamental form of translating solitons, convexity of the bounded level set for any solution to \eqref{tsl} could imply convexity of the corresponding sublevel set with a restriction of the mean curvature of the level set. Then it is able to solve a class of Dirichlet problems in smooth bounded convex domains, see Theorem \ref{Dirichlet}.
We construct auxiliary functions to seek out a sequence of convex solutions $\{u_k\}$ to \eqref{tsl} in different bounded domains which is asymptotic to prescribing function at infinity. Convexity of $u_k$ could enable us to get the uniform bound for Hessian of $u_k$, which is crucial to show that their limit $u$ is the desired function in Theorem \ref{Vcu}.

{\bf Acknowledgement} The author would like to express his
sincere thanks to Yuanlong Xin for his continuous support and valuable comments on this paper.

\section{Translating solitons from variational view}

In this section, we always suppose that $\Om$ is a bounded domain in $\R^n$.
Let $M=\{(x,u(x))\mid x\in\Om\}$ be a weakly spacelike hypersurface in $\R^{n+1}_1$ with volume element $d\mu=\sqrt{1-|Du|^2}dx$. We define a functional $F$ by 
\begin{equation}\aligned\label{Var}
F(M)=\int_{M}e^{-x_{n+1}}d\mu=\int_\Om e^{-u(x)}\sqrt{1-|Du|^2}dx.
\endaligned
\end{equation}
If $u$ is smooth and $|Du|<1$ in $\Om$, then $M$ is a Riemannian manifold with the metric $g_{ij}dx_idx_j$, where $g_{ij}=\de_{ij}-u_iu_j$. Let $\lan\ ,\ \ran$ be indefinite inner product in $\R^{n+1}_1$ endowed by the Lorentz metric $\bar{g}$ (This definition does not conflict with $n$-dimensional Euclidean inner product $\lan\ ,\ \ran$ in $\S$ 1). The normal vector field $\nu=\f1{\sqrt{1-|Du|^2}}\big(\sum_{i=1}^nu_iE_i+E_{n+1}\big)$ satisfies $\lan\nu,\nu\ran=-1$, where $E_1,\cdots,E_{n+1}$ is the unit natural basis of $\R^{n+1}_1$.

If $M_s$ is a variation of $M$ with $s\in(-1,1)$ and $M_0=M$, where the variation vector filed $M_0'=f\nu$ and $f$ is smooth with $f\mid_{\p M_0}\equiv0$. By a calculation (see \cite{X} for Euclidean case or \eqref{FWwe} for $W\equiv0$), we have
\begin{equation}\aligned\label{v1st}
\f{d}{ds}F(M_s)\bigg|_{s=0}=\int_{M}\big(-fH-f\lan\nu,E_{n+1}\ran\big)e^{-x_{n+1}}d\mu,
\endaligned
\end{equation}
where $H=\div\left(\f{Du}{\sqrt{1-|Du|^2}}\right)$.
We will say that $M$ is a \emph{critical point} for the functional $F$ if it is critical with respect to all normal variations in $M$. Hence,
$M$ is a critical point for $F$ if and only if the mean curvature of $M$ satisfies
\begin{equation}\aligned\label{HnuE}
H=-\lan\nu,E_{n+1}\ran.
\endaligned
\end{equation}
\eqref{HnuE} is just \eqref{tsl}.
Let $g^{ij}=\de_{ij}+\f{u_iu_j}{1-|Du|^2}$, then \eqref{tsl} can be rewritten as
\begin{equation}\aligned\label{ts}
\sum_{i,j}g^{ij}u_{ij}=\sum_{i,j}\left(\de_{ij}+\f{u_iu_j}{1-|Du|^2}\right)u_{ij}=1.
\endaligned
\end{equation}

Let $\varphi$ be a weakly spacelike function on $\overline{\Om}$ and $\mathcal{C}(\varphi,\Om)$ be a set defined by
$$\{w\in C^{0,1}(\Om)|\ w=\varphi\ \mathrm{on}\ \p\Om,\ \mathrm{and}\ |Dw|\le1\ a.e.\ \mathrm{in}\ \Om\}.$$
For any $\ep>0$, denote $\Om_\ep=\{x\in\Om\big|\ \inf_{y\in\p\Om}|x-y|\ge\ep\}$ and $\Om^\ep=\{x\in\R^n\big|\ \inf_{y\in\overline{\Om}}|x-y|\le\ep\}$. By the boundedness of $\overline{\Om}$, $\Om_\ep$ and $\Om^\ep$ are both closed sets.
Let $\r$ be a smooth function with compact support in $B_1(0)\subset\R^n$ and $\int_{\R^n}\r(x)dx=1$. Let $w_\ep$ be a mollifier of weakly spacelike function $w\in\mathcal{C}(\varphi,\Om)$ defined by
\begin{equation}\aligned
w_\ep(x)=(w*\r_\ep)(x)\triangleq\int_{\R^n}\r(y)w(x-\ep y)dy\qquad \mathrm{for}\ x\in\Om_\ep.
\endaligned
\end{equation}
Then
\begin{equation}\aligned\label{Dw}
\big|Dw_\ep(x)\big|^2=&\sum_i\left|\int_{\R^n}\r(y)\p_{x_i}w(x-\ep y)dy\right|^2\le\sum_i\int_{\R^n}\r(y)dy\int_{\R^n}\r(y)\big|\p_{x_i}w(x-\ep y)\big|^2dy\\
=&\int_{\R^n}\r(y)\big|Dw(x-\ep y)\big|^2dy\le\int_{\R^n}\r(y)dy=1.
\endaligned
\end{equation}
Moreover, $w_\ep\rightarrow w$ uniformly and $Dw_\ep\rightarrow Dw$ a.e. in any compact set $K\subset\Om$.

\begin{theorem}\label{uni}
Let $\Om$ be a bounded domain in $\R^n$ and $u$ be a smooth strictly spacelike solution to \eqref{tsl} in $\overline{\Om}$. Set $M=\{(x,u(x))\big|\ x\in\Om\}$, then for any bounded weakly spacelike hypersurface $\Si\subset \R^{n+1}_1$ with $\p\Si=\p M$, one has
\begin{equation}\aligned
\int_\Si e^{-x_{n+1}}dV_\Si\le\int_M e^{-x_{n+1}}dV_M,
\endaligned
\end{equation}
where the above inequality attains equality if and only if $\Si=M$.
\end{theorem}
\begin{proof}
There is a domain $\widetilde{\Om}$ with $\Om\subset\subset \widetilde{\Om}$ such that $u$ can be extended to $\widetilde{\Om}$ smoothly, and we may assume $|Du|<1$ in $\overline{\widetilde{\Om}}$. $\Si$ can be written as $\{(x,w(x))\big|\ x\in\Om\}$ for some $w\in\mathcal{C}(u,\Om)$. Since $w\big|_{\p\Om}=u\big|_{\p\Om}$, then we extend $w$ to $\widetilde{\Om}$ with $w\big|_{\widetilde{\Om}\setminus\Om}=u\big|_{\widetilde{\Om}\setminus\Om}$. Clearly, there is a $\ep_0>0$ such that $\Om^{\ep_0}\subset \widetilde{\Om}$. Since $u$ is smooth in $\Om^{\ep_0}$, then there exists a constant $C$ depending only on $n$ and $\sup_{\Om^{\ep_0}}|D^2u|$, such that for any $0<\ep\le\f{\ep_0}2$ and $x\in\Om^\ep$ one has
\begin{equation}\aligned\label{D2u}
|D(u*\r_\ep)(x)-Du(x)|=\left(\sum_i\bigg|\int_{\R^n}\r(y)\big(\p_{x_i}u(x-\ep y)-\p_{x_i}u(x)\big)dy\bigg|^2\right)^{\f12}\le C\ep.
\endaligned
\end{equation}
Set $\widetilde{w}_\ep=(w-u)*\r_\ep$ and $w_\ep=(1-C\ep)(\widetilde{w}_\ep+u)$. Then $w_\ep$ is a smooth function in $\Om^\f{\ep_0}2$ and $w_\ep(x)=(1-C\ep)u(x)$ for any $x\in\p\Om^\ep$ and $0<\ep\le\f{\ep_0}2$. Moreover $w_\ep\rightarrow w$ uniformly and $Dw_\ep\rightarrow Dw$ a.e. in $\overline{\Om}$.  By \eqref{D2u}, in $\Om^\ep$ one obtains
\begin{equation}\aligned
|Dw_\ep|=&(1-C\ep)|D(w*\r_\ep-u*\r_\ep+u)|\\
\le&(1-C\ep)|D(w*\r_\ep)|+(1-C\ep)|D(u*\r_\ep-u)|\\
\le&1-C\ep+C\ep(1-C\ep)=1-C^2\ep^2.
\endaligned
\end{equation}

Let $M_\ep=\{(x,(1-C\ep)u(x))\big|\ x\in\Om^\ep\}$ and $\Si_\ep=\{(x,w_\ep(x))\big|\ x\in\Om^\ep\}$, then they are both smooth strictly spacelike hypersurfaces with $\p M_\ep=\p\Si_\ep$.
Let $D_\ep$ be the domain enclosed by $M_\ep$ and $\Si_\ep$. Let $v=\sqrt{1-|Du|^2}$ and $Y$ be a vector field in $M_\ep$ defined by
$$Y=\sum_{i=1}^n\f{u_i}{v}e^{-x_{n+1}}E_i+\f{e^{-x_{n+1}}}{v}E_{n+1}.$$
Viewing $u_i$ and $v$ as functions of $x_1,\cdots,x_n$ and translating $Y$ to $\Si_\ep$ along the $x_{n+1}$ axis. Then we obtain a vector field in $D_\ep$, denoted by $Y$, too. Let $\overline{\na}$ and $\overline{\div}$ be Levi-Civita connection and divergence on $\R^{n+1}_1$ with the Lorentz metric $\bar{g}$, respectively.
From \cite{ED}, one has
\begin{equation}\aligned
\overline{\div}(Y)=&\sum_i\lan\overline{\na}_{E_i}Y,E_i\ran-\lan\overline{\na}_{E_{n+1}}Y,E_{n+1}\ran\\
=&\sum_i\lan E_i,[E_i,Y]\ran-\lan E_{n+1},[E_{n+1},Y]\ran\\
=&\sum_i\p_{x_i}\left(\f{u_i}{v}e^{-x_{n+1}}\right)+\p_{x_{n+1}}\left(\f{e^{-x_{n+1}}}{v}\right)\\
=&\sum_i\p_{x_i}\left(\f{u_i}{v}\right)e^{-x_{n+1}}-\f1{v}e^{-x_{n+1}}.
\endaligned
\end{equation}
Let $\nu_\ep,\nu_{\Si_\ep}$ be the timelike future-pointing unit normal vectors of $M_\ep,\Si_\ep$  respectively, then by Gauss formula (see \cite{ED} for example), up to a minus sign we have
\begin{equation}\aligned\label{Gaussf}
\int_{D_\ep}\overline{\div}(Y)=\int_{M_\ep}\lan Y,\nu_\ep\ran dV_{M_\ep}-\int_{\Si_\ep}\lan Y,\nu_{\Si_\ep}\ran dV_{\Si_\ep}.
\endaligned
\end{equation}
In fact, let the orientation of $\p D_\ep$ direct timelike future-pointing in $M_\ep\setminus\Si_\ep$ and direct timelike past-pointing in $\Si_\ep\setminus M_\ep$. Let $(\g_1,\cdots,\g_{n+1})$ with $\g_{n+1}^2-\sum_i\g_i^2=1$ be the timelike unit normal vector of $\p D_\ep\setminus(M_\ep\cap\Si_\ep)$ with respect to the orientation of $\p D_\ep$. Set $(\g_1,\cdots,\g_{n+1})$ be an arbitrary constant vector in $M_\ep\cap\Si_\ep$, then we have
\begin{equation}\aligned\nonumber
&\int_{D_\ep}\overline{\div}(Y)
=-\int_{D_\ep}\left(\sum_i\p_{x_i}\left(\f{u_i}{v}e^{-x_{n+1}}\right)+\p_{x_{n+1}}\left(\f{e^{-x_{n+1}}}{v}\right)\right)dx_1\wedge\cdots\wedge dx_{n+1}\\
=&\sum_i(-1)^{i}\int_{\p D_\ep}\f{u_i}{v}e^{-x_{n+1}}dx_1\wedge\cdots \widehat{dx_i}\cdots\wedge dx_{n+1}+(-1)^{n+1}\int_{\p D_\ep}\f{e^{-x_{n+1}}}{v}dx_1\wedge\cdots\wedge dx_{n}\\
=&\sum_i\int_{\p D_\ep}\f{u_i}{v}e^{-x_{n+1}}\g_idV_{\p D_\ep}-\int_{\p D_\ep}\f{e^{-x_{n+1}}}{v}\g_{n+1}dV_{\p D_\ep}\\
=&\int_{M_\ep}\lan Y,\nu_\ep\ran dV_{M_\ep}+\int_{\Si_\ep}\lan Y,-\nu_{\Si_\ep}\ran dV_{\Si_\ep}.
\endaligned
\end{equation}
If $\xi$ and $\e$ are timelike future-pointing vectors in $\R^{n+1}_1$, then reversed Cauchy-Schwarz inequality implies (see\cite{O} for example) $$-\lan\xi,\e\ran\ge\sqrt{|\lan\xi,\xi\ran|}\sqrt{|\lan\e,\e\ran|}.$$
Substituting it into \eqref{Gaussf} gives
\begin{equation}\aligned\label{DIVY}
\int_{D_\ep}\overline{\div}(Y)-\int_{M_\ep}\lan Y,\nu_\ep\ran dV_{M_\ep}=&-\int_{\Si_\ep}\lan Y,\nu_{\Si_\ep}\ran dV_{\Si_\ep}\\
\ge&\int_{\Si_\ep}\sqrt{|\lan Y,Y\ran|} dV_{\Si_\ep}=\int_{\Si_\ep} e^{-x_{n+1}}dV_{\Si_\ep}.
\endaligned
\end{equation}
The inequality \eqref{DIVY} arrives at equality if and only if $Y$ parallels $\nu_{\Si_\ep}$. By \eqref{tsl}, $\overline{\div}(Y)=0$ in $\Om$. Since $u$ is strictly spacelike, then $\int_{D_\ep}\overline{\div}(Y)\rightarrow0$ as $\ep\rightarrow0$. Combining  $\lan Y,\nu_\ep\ran\rightarrow-e^{-x_{n+1}}$ and \eqref{DIVY}, we obtain the desired result.
\end{proof}

For any $C^2$ function $w$ with $|Dw|<1$, we define a differential operator by
$$\mathcal{L}w=\sum_{i,j}\left(\de_{ij}+\f{w_iw_j}{1-|Dw|^2}\right)w_{ij}.$$
\begin{lemma}\label{max}
Let $u,\overline{u},\underline{u}$ be three $C^2$ strictly spacelike functions satisfying $\mathcal{L}u=1,\mathcal{L}\overline{u}\le1,\mathcal{L}\underline{u}\ge1$ in $\Om$, then
\begin{equation}\aligned
&u(x)-\overline{u}(x)\le\sup_{y\in\p\Om}\big(u(y)-\overline{u}(y)\big)\qquad \mathrm{for}\ x\in\Om,\\
&u(x)-\underline{u}(x)\ge\inf_{y\in\p\Om}\big(u(y)-\underline{u}(y)\big)\qquad \mathrm{for}\ x\in\Om.
\endaligned
\end{equation}
\end{lemma}
\begin{proof}
Let $w=u-\overline{u}$, then
\begin{equation}\aligned
0\le&\mathcal{L}u-\mathcal{L}\overline{u}=\left(\de_{ij}+\f{u_iu_j}{1-|Du|^2}\right)w_{ij}+
\left(\f{u_iu_j}{1-|Du|^2}-\f{\overline{u}_i\overline{u}_j}{1-|D\overline{u}|^2}\right)\overline{u}_{ij}\\
=&\left(\de_{ij}+\f{u_iu_j}{1-|Du|^2}\right)w_{ij}+
\f{(|Du|^2-|D\overline{u}|^2)u_iu_j+(1-|Du|^2)(u_iu_j-\overline{u}_i\overline{u}_j)}{(1-|Du|^2)(1-|D\overline{u}|^2)}\overline{u}_{ij}\\
=&\left(\de_{ij}+\f{u_iu_j}{1-|Du|^2}\right)w_{ij}+
\f{u_iu_j(u_k+\overline{u}_k)w_k+(1-|Du|^2)(w_iu_j-\overline{u}_iw_j)}{(1-|Du|^2)(1-|D\overline{u}|^2)}\overline{u}_{ij}.
\endaligned
\end{equation}
By the maximum principle of elliptic equations, we have $w(x)\le\sup_{y\in\p\Om}w(y)$ for each $x\in\Om$. Clearly, one could prove the second inequality in Lemma \ref{max} by the same method.
\end{proof}

Let $W$ be a continuous function in $\Om$. We define a functional $F_{W,\Om}$ on a function $w\in\mathcal{C}(\varphi,\Om)$ by
$$F_{W,\Om}(w)=\int_\Om e^{-w}\left(\sqrt{1-|Dw|^2}+W\right)dx.$$
Denote $F_{0,\Om}(\cdot)$ by $F_\Om(\cdot)$ for simplicity.
The Dirichlet problem
\begin{equation}\label{evo}\left\{\begin{split}
&\div\left(\f{Dw}{\sqrt{1-|Dw|^2}}\right)-\f1{\sqrt{1-|Dw|^2}}=W\qquad \mathrm{in}\ \Om,\\
&w(x)=\varphi(x)\qquad \mathrm{for}\quad x\in\p\Om,\\
\end{split}\right.
\end{equation}
raises from the Euler-Lagrange equation of the variational problem
$\sup_{z\in \mathcal{C}(\varphi,\Om)}F_{W,\Om}(z)$. In fact, for any $\e\in C_c^\infty(\Om)$ a simply calculation gives
\begin{equation}\aligned\label{FWwe}
\f{d}{dt}\bigg|_{t=0}F_{W,\Om}(w+t\e)=&\int_\Om\left(-\e e^{-w}\left(\sqrt{1-|Dw|^2}+W\right)-e^{-w}\f{Dw\cdot D\e}{\sqrt{1-|Dw|^2}}\right)\\
=&\int_\Om\e e^{-w}\left(\div\left(\f{Dw}{\sqrt{1-|Dw|^2}}\right)-\f1{\sqrt{1-|Dw|^2}}-W\right).
\endaligned
\end{equation}

\begin{lemma}\label{comp}
Let $u$ be a strictly spacelike function to \eqref{ts} in $\Om$ with $u\big|_{\p\Om}=\varphi$. Let $\varphi_1,\varphi_2$ be weakly spacelike functions in $\overline{\Om}$ and $W_1$ be a nonpositive continuous function and $W_2$ be a nonnegative continuous function in $\Om$. If $w_i\in \mathcal{C}(\varphi_i,\Om)$ and $F_{W_i,\Om}(w_i)=\sup_{w\in \mathcal{C}(\varphi_i,\Om)}F_{W_i,\Om}(w)$ for $i=1,2$, then
\begin{equation}\aligned\label{comp2}
&u(x)\le w_1(x)+\sup_{\p\Om}\big(\varphi-\varphi_1\big)\qquad \mathrm{for}\ x\in\Om\\
&u(x)\ge w_2(x)+\inf_{\p\Om}\big(\varphi-\varphi_2\big)\qquad \mathrm{for}\ x\in\Om.
\endaligned
\end{equation}
\end{lemma}
\begin{proof}
Let $C=\sup_{\p\Om}\big(\varphi-\varphi_1\big)$, $\ep$ be a small positive constant, $w^*=w_1+C+\ep$ and $\Om^+=\{x\in\Om|\ u(x)>w^*(x)\}$. Assume $\Om^+\neq\emptyset$.

Clearly, $\overline{\Om^+}\subset\Om$ and $u(x)=w^*(x)$ on $\p\Om^+$. By $F_{W_1,\Om^+}(w^*)\ge F_{W_1,\Om^+}(u)$, we conclude
\begin{equation}\aligned
\int_{\Om^+}e^{-w^*}\sqrt{1-|Dw^*|^2}dx=&F_{W_1,\Om^+}(w^*)-\int_{\Om^+}e^{-w^*}W_1dx\\
\ge& F_{W_1,\Om^+}(u)-\int_{\Om^+}e^{-w^*}W_1dx\\
=&\int_{\Om^+}e^{-u}\sqrt{1-|Du|^2}dx+\int_{\Om^+}\left(e^{-u}-e^{-w^*}\right)W_1dx\\
\ge&\int_{\Om^+}e^{-u}\sqrt{1-|Du|^2}dx.
\endaligned
\end{equation}
Theorem \ref{uni} implies $F_{\Om^+}(w^*)\le F_{\Om^+}(u)$, then $F_{\Om^+}(w^*)=F_{\Om^+}(u)$ and $W_1\equiv0$. Hence $\Om^+$ is empty by Theorem \ref{uni}. Letting $\ep\rightarrow0$ yields the first inequality in \eqref{comp2}. The second inequality in \eqref{comp2} could be showed similarly.
\end{proof}
\begin{remark}
Lemma \ref{comp} can be seen as a weak version of Lemma \ref{max}, if we set $W_1=\f{\mathcal{L}\overline{u}-1}{\sqrt{1-|D\overline{u}|^2}}$ and $W_2=\f{\mathcal{L}\underline{u}-1}{\sqrt{1-|D\underline{u}|^2}}$.
\end{remark}

\section{Barrier functions and applications}

For $K\in\R$, let
\begin{equation}\aligned\label{wwt}
w_K(t)\triangleq\int_0^t\f K{\sqrt{s^{2n-2}+K^2}}ds\ \ \ \mathrm{and}\ \ \ \widetilde{w}_K(t)\triangleq\int_0^t\f K{\sqrt{s^{2n}+K^2}}ds\qquad \mathrm{for}\ t\ge0.
\endaligned
\end{equation}
Then
\begin{equation}\aligned\label{Lw}
\mathcal{L}w_K=\f{w_K''}{1-(w_K')^2}+\f{n-1}tw_K'=t^{1-n}\sqrt{1-(w_K')^2}\left(\f{t^{n-1}w_K'}{\sqrt{1-(w_K')^2}}\right)'=0,
\endaligned
\end{equation}
and
\begin{equation}\aligned\label{Ltw}
\mathcal{L}\widetilde{w}_K=\f{\widetilde{w}_K''}{1-(\widetilde{w}_K')^2}+\f{n-1}t\widetilde{w}_K'=-\f K{t\sqrt{t^{2n}+K^2}}.
\endaligned
\end{equation}
Moreover, $\lim_{K\rightarrow\pm\infty}w_K(t)=\lim_{K\rightarrow\pm\infty}\widetilde{w}_K(t)=\pm t$.

Now we give an existence theorem for an ODE arising from \eqref{ts}.
\begin{theorem}\label{test}
For any constant $C\in(-r,r)$ and $r>0$, the following ODE:
\begin{equation}\label{ODE}\left\{\begin{split}
&\f{\phi''}{1-(\phi')^2}+\f{n-1}t\phi'=1\qquad\ \mathrm{for}\ t\in (0,r),\\
&\phi(r)=C,\ \phi(0)=0\ \mathrm{and}\ |\phi'|<1,\\
\end{split}\right.
\end{equation}
has a unique smooth solution $\phi_0$ in $(0,r)$. Furthermore, if $r^2\le(n-1)C$ and $K_1>0$ is a constant with $w_{K_1}(r)=C$, then $$\f Crt\le\phi_0(t)\le w_{K_1}(t)\quad for\ \ t\in[0,r].$$
If $C<0$, $r<1$, $K_2<0$ is a constant with $\widetilde{w}_{K_2}(r)=C$ and $K_2^2\ge\f{r^{2n+2}}{1-r^2}$, then $$\widetilde{w}_{K_2}(t)\le\phi_0(t)\le \f{C}rt\quad for\ \ t\in[0,r].$$
\end{theorem}
\begin{proof}
We consider a family of approximation equations:
\begin{equation}\label{ODE*}\left\{\begin{split}
&\f{\phi''}{1-(\phi')^2}+\f{n-1}t\phi'=1\qquad \ \mathrm{for}\ t\in [\ep,r),\\
&\phi(r)=C,\ \phi(\ep)=0\ \mathrm{and}\ |\phi'|<1.\\
\end{split}\right.
\end{equation}
Clearly, \eqref{ODE*} has a smooth solution $\phi_\ep$ on $[\ep,r)$ for $0<\ep<r-|C|$. For any fixed $\de\in(0,r)$ there is a sequence $\{\ep_i\}$ such that $\phi_{\ep_i}\rightarrow\phi_0$ uniformly in $[\de,r)$ as $\ep_i\rightarrow0$.
Then $\phi_0\in C^{0,1}((0,r))$, $\phi_0(r)=C$, $\phi_0(0)=0$ and $|\phi'_0|\le1$ a.e..

If there is a subsequence $\ep_{i_l}\rightarrow0$ (we denoted it by $\ep_i$) such that $\phi'_{\ep_i}(t_0)\rightarrow\a$ for some fixed $t_0\in(0,r)$, then integrating \eqref{ODE*} for $\phi'_\ep$ with sufficiently small $0<\ep<t_0$ gives
\begin{equation}\aligned\label{phide}
\f12\log\f{1+\phi'_\ep(t)}{1-\phi'_\ep(t)}\bigg|_{t_0}^t+(n-1)\int_{t_0}^t\f{\phi'_\ep(s)}sds=t-t_0\quad \mathrm{for}\ t\in[\ep,r).
\endaligned
\end{equation}
If $\a=\pm1$, then letting $\ep_i\rightarrow0^+$ gets $\phi'_{\ep_i}(t)\rightarrow\pm1$ for any $t\in(0,r)$. Set $$f_\ep(t)=\f{1+\phi'_\ep(t_0)}{1-\phi'_\ep(t_0)}e^{2(t-t_0)-2(n-1)\int^t_{t_0}\f{\phi'_\ep(s)}sds},$$ then \eqref{phide} implies
$\f{1+\phi'_\ep(t)}{1-\phi'_\ep(t)}=f_\ep(t)$ and
\begin{equation}\aligned\label{intphi}
\phi_\ep(t)-\phi_\ep(t_0)=\int^t_{t_0}\f{f_\ep(s)-1}{f_\ep(s)+1}ds.
\endaligned
\end{equation}
Thus one has $\phi_0(t)-\phi_0(t_0)=\pm(t-t_0)$ which contradicts with $\phi_0(r)-\phi_0(0)=C$. Therefore, $\phi'_{\ep_i}(t_0)\rightarrow\a\in(-1,1)$, then by \eqref{intphi} and $\phi_0\in C^{0,1}((0,r))$, we obtain
\begin{equation}\aligned
\phi_0(t)-\phi_0(t_0)=
\int^t_{t_0}\f{\a-1+(1+\a)e^{2(s-t_0)-2(n-1)\int^s_{t_0}\f{\phi'_0(p)}pdp}}{1-\a+(1+\a)e^{2(s-t_0)-2(n-1)\int^s_{t_0}\f{\phi'_0(p)}pdp}}ds\qquad \mathrm{for}\ t>0.
\endaligned
\end{equation}
Hence $\phi'_0(t)$ exists everywhere in $(0,r)$, and one gets
\begin{equation}\aligned\label{phida}
\f12\log\f{1+\phi'_0(t)}{1-\phi'_0(t)}-\f12\log\f{1+\a}{1-\a}+(n-1)\int_{t_0}^t\f{\phi'_0(s)}sds=t-t_0.
\endaligned
\end{equation}
Then $\phi'_0(t)$ is continuous and $\phi''_0(t)$ exists everywhere in $(0,r)$. Now \eqref{phida} implies
$$\phi''_0=\left(1-(\phi'_0)^2\right)\left(1-\f{n-1}t\phi'_0\right).$$
The above equation shows that $\phi_0$ is our desired smooth solution to \eqref{ODE}. By Lemma \ref{max}, we know the uniqueness of the smooth solution to \eqref{ODE}.

For $C\ge\f{r^2}{n-1}$, one has
$$\mathcal{L}\big(\f {Ct}r\big)=\f{n-1}t\f Cr\ge\f{(n-1)C}{r^2}\ge1=\mathcal{L}\phi_0.$$
Let $K_1>0$ be a constant with $w_{K_1}(r)=C$. Combining \eqref{Lw} and Lemma \ref{max} gives $\f{Ct}r\le\phi_0(t)\le w_{K_1}(t)$ on $[0,r]$.

For $C<0$ and $r<1$, one selects a negative constant $K_2$ satisfying $\widetilde{w}_{K_2}(r)=C$. If $K_2^2\ge\f{r^{2n+2}}{1-r^2}$, then \eqref{Ltw} yields
\begin{equation}\aligned
\mathcal{L}\widetilde{w}_{K_2}\ge-\f{K_2}{r\sqrt{r^{2n}+K_2^2}}\ge1=\mathcal{L}\phi_0.
\endaligned
\end{equation}
By Lemma \ref{max}, we have $\widetilde{w}_{K_2}(t)\le\phi_0(t)\le \f{Ct}r$ on $[0,r]$.
\end{proof}

Let $\Om$ be a bounded domain in $\R^n$ and $\varphi$ be a weakly spacelike function on $\overline{\Om}$. Let $F_{\Om}(\cdot)$ be defined previously. Clearly, $\sup_{w\in \mathcal{C}(\varphi,\Om)}F_\Om(w)$ is bounded, which implies there exists a sequence $\{u_k\}\subset\mathcal{C}(\varphi,\Om)$ such that $\lim_{k\rightarrow\infty}F_\Om(u_k)=\sup_{w\in \mathcal{C}(\varphi,\Om)}F_\Om(w)$.
The equicontinuity of $\mathcal{C}(\varphi,\Om)$ then gives a uniformly convergent subsequence $u_{k_i}$ of maximizing sequence $u_k$ with $u_{k_i}\rightarrow u_0\in \mathcal{C}(\varphi,\Om)$. By the Appendix,
$F_\Om(u_0)=\sup_{w\in \mathcal{C}(\varphi,\Om)}F_\Om(w)$ is maximal.

The next result for translating solitons is similar to hypersurfaces with bounded mean curvature in Minkowski space, see Theorem 3.2 in \cite{BS}
for example. However, the mean curvature of any translating soliton is unbounded (see Proposition \ref{unbd H}).
\begin{lemma}\label{line}
If $u$ is a weakly spacelike solution to the variational functional $F$. Let $x_0,x_1\in\overline{\Om}$ with the open line segment $\overline{x_0x_1}\subset\Om$ such that
\begin{equation}\aligned\label{u0xt}
u(x_t)=u(x_0)+t|x_0-x_1|\qquad for\ t\in[0,1],
\endaligned
\end{equation}
where $x_t=x_0+t(x_1-x_0)$. Then \eqref{u0xt} holds for all $t\in\R$ such that $x_t\in\Om$ and $\overline{x_0x_t}\subset\Om$.
\end{lemma}
\begin{proof}
We prove it by following the steps of the proof for Theorem 3.2 in \cite{BS}. Suppose that \eqref{u0xt} holds for $\forall t\in[-\f14,1]$,  $u(x_{-\f12})>u(x_0)-\f12|x_0-x_1|$, $|x_0-x_1|<1$ and the Euclidean ball $B_{2|x_0-x_1|}(x_1)\subset\subset\Om$. Let $B=B_{\f12|x_0-x_1|}(x_0)$, and let $C_1$ and $C_2$ denote the backward light cones with apexes at $(x_1,u(x_1))$ and $(x_0,u(x_0))$ respectively. Then we have
$$u(x)\ge C_1(x)>C_2(x)\qquad  \mathrm{for\ any\ } x\in B\setminus\{x_t|\ t\in[-\f12,0]\}.$$
Combining $u(x_{-\f12})>u(x_0)-\f12|x_0-x_1|$ gives
$$u(x)>C_2(x)\qquad \mathrm{for\ any\ } x\in\p B.$$
Let $\widetilde{w}_K$ be defined as \eqref{wwt} with sufficiently small negative number $K$, then
$$u(x)>\widetilde{w}_K(|x-x_0|)+u(x_0)>C_2(x)\qquad \mathrm{for\ any\ } x\in\p B.$$
But $\widetilde{w}_K$ is strictly spacelike away from $x_0$, so $\widetilde{w}_K(|x_t-x_0|)+u(x_0)>u(x_t)$ for $t\in[-\f14,0)$, and this
contradicts Lemma \ref{comp} applied with $\Om=B\setminus\{x_0\}$.

A completely analogous argument holds if \eqref{u0xt} fails for $t>1$.
\end{proof}

Let $u$ be a smooth function satisfying \eqref{ts} and $g_{ij},g^{kl}$ be defined as in $\S$ 2, then $$g^{ij}u_{ijk}-g^{ip}(\p_kg_{pq})g^{qj}u_{ij}=0.$$
Here we have adopted Einstein convention of summation over repeated indices. Denote $v=\sqrt{1-|Du|^2}$ as before. Hence
\begin{equation}\aligned
g^{ij}(|Du|^2)_{ij}=&2g^{ij}(u_{ki}u_{kj}+u_{ijk}u_k)\\
=&2g^{ij}u_{ki}u_{kj}-2g^{ip}(u_pu_{kq}+u_qu_{kp})u_kg^{qj}u_{ij}\\
=&2g^{ij}u_{ki}u_{kj}-4g^{ip}u_pu_{kq}u_kg^{qj}u_{ij}\\
=&2g^{ij}u_{ki}u_{kj}-\f4{v^2}u_{kq}u_kg^{qj}u_iu_{ij}\\
=&2g^{ij}u_{ki}u_{kj}-\f1{v^2}g^{qj}(|Du|^2)_q(|Du|^2)_j,\\
\endaligned
\end{equation}
i.e.,
\begin{equation}\aligned\label{bge}
g^{ij}(|Du|^2)_{ij}+\f1{v^2}g^{ij}(|Du|^2)_i(|Du|^2)_j=2g^{ij}u_{ki}u_{kj}\ge0.\\
\endaligned
\end{equation}

\section{Convexity and Dirichlet problem}

In this section, we always suppose that $M=\{(x,u(x))|\ x\in\R^n\}$ is a smooth strictly spacelike hypersurface in $\R^{n+1}_1$ with $u$ satisfying \eqref{tsl}. Let $\na$, $\div_M$ and $\De$ be Levi-Civita connection, divergence and Laplacian operator on $M$ with the induced metric from $\R^{n+1}_1$, respectively. Inspired by the 'drift Laplacian' on self-shrinkers, which was introduced by Colding-Minicozzi \cite{CM1}, we define a second order differential operator $L$ by
$$Lf=e^{u}\div_M(e^{-u}\na f)\qquad \mathrm{for}\ f\in C^2(\R^n).$$
Let $E_1,\cdots,E_{n+1}$ be the unit natural basis of $\R^{n+1}_1$, and $\p_jf=\f{\p f}{\p x_j}$. Denote $\nu$ be the unit normal vector field of $M$: $\f1{\sqrt{1-|Du|^2}}\big(\sum_{i=1}^nu_iE_i+E_{n+1}\big)$.
Since
\begin{equation}\aligned\label{Defnaf}
\na f=&\sum_{i=1}^nf_iE_i+\left\lan\sum_{i=1}^nf_iE_i,\nu\right\ran\nu=\sum_{i=1}^nf_iE_i+\f1{\sqrt{1-|Du|^2}}\left(\sum_{i=1}^nu_if_i\right)\nu\\
=&\sum_{i=1}^n\left(f_i+\f1{1-|Du|^2}\left(\sum_{j=1}^nu_jf_j\right)u_i\right)E_i+\f1{1-|Du|^2}\left(\sum_{i=1}^nu_if_i\right)E_{n+1},
\endaligned
\end{equation}
then one has
\begin{equation}\aligned
\na u=\f1{1-|Du|^2}\sum_{i=1}^nu_iE_i+\f{|Du|^2}{1-|Du|^2}E_{n+1},
\endaligned
\end{equation}
and
\begin{equation}\aligned\label{naunaf}
\lan\na u,\na f\ran=\f1{1-|Du|^2}\sum_{i=1}^nu_if_i.
\endaligned
\end{equation}
\eqref{Defnaf} and \eqref{naunaf} imply
\begin{equation}\aligned\label{Lf}
Lf=&\De f-\lan\na u,\na f\ran=\De f-\f1{1-|Du|^2}\sum_{i=1}^nu_if_i=\De f+\lan E_{n+1},\na f\ran.
\endaligned
\end{equation}

\begin{proposition}\label{unbd H}
The manifold $M=\{(x,u(x))|\ x\in\R^n\}$ has unbounded mean curvature.
\end{proposition}
\begin{proof}
Suppose that the mean curvature $H$ is bounded, then there exists a positive constant $\delta>0$ such that $v=\sqrt{1-|Du|^2}\ge\de$. Recall that $g_{ij}=\de_{ij}-u_iu_j$ and $g^{ij}=\de_{ij}+\f{u_iu_j}{1-|Du|^2}$. Let $g=\det g_{ij}$, then by \eqref{tsl} one has
\begin{equation}\aligned\label{4Deu}
\De u=\f1{\sqrt{g}}\sum_{i,j}\p_i\big(g^{ij}\sqrt{g}\p_ju\big)=\f1{\sqrt{g}}\sum_{i}\p_i\left(\f{u_i}v\right)=\f1{v^2}.
\endaligned
\end{equation}
Let $B_r$ be a ball in $\R^n$ with radius $r$ and centered at the origin. Let $\e$ be a nonnegative Lipschitz function with $\e\big|_{B_r}\equiv1$, $|D\e|\le\f1r$ and $\e\big|_{\R^n\setminus B_{2r}}\equiv0$. For any $p=(x,u(x))\in M$, we set $\e(x,u(x))=\e(x)$. Denote $\omega_n$ be the volume of $n$-unit ball. Noting $g^{ij}u_j=\f{u_i}{1-|Du|^2}$, by \eqref{tsl} we have
\begin{equation}\aligned
\omega_nr^n\le&\int_{B_r}\f1vdx\le\int_{\R^n}\f{\e}{v^2}vdx=\int_{\R^n}\e\p_i\big(g^{ij}\sqrt{g}\p_ju\big)dx=-\int_{\R^n}\p_i\e g^{ij}\p_ju\sqrt{g}dx\\
=&-\int_{\R^n}\f{Du\cdot D\e}{v^2}vdx\le\int_{\R^n}\f{|D\e|}vdx\le\f1{r\de}\int_{B_{2r}}dx=\f1\de\omega_n2^nr^{n-1}.
\endaligned
\end{equation}
Selecting sufficiently large $r$, we get the desired contradiction.
\end{proof}
\begin{proposition}
If $u\ge0$ and $\int_Mu^2e^{-u}d\mu<\infty$, then $M$ is noncomplete.
\end{proposition}
\begin{proof}
Suppose that $M$ is complete, then we could define a nonnegative Lipschitz function $\e$ in $M$ satisfying $\e\big|_{D_r}\equiv1$, $|\na\e|\le\f1r$ and $\e\big|_{M\setminus D_{2r}}\equiv0$ for any $r>0$. Here, $D_r$ is a geodesic ball with radius $r$ in $M$. By \eqref{4Deu}, we have $$Lu=\De u-|\na u|^2=\f1{v^2}-\f{|Du|^2}{v^2}=1,$$ then
\begin{equation}\aligned
\int_{D_r}ue^{-u}\le&\int_Mu\e^2e^{-u}=\int_Mu\e^2e^{-u}Lu=-\int_M\na u\cdot\na(u\e^2)e^{-u}\\
=&-\int_M\e^2|\na u|^2e^{-u}-2\int_Mu\e\na u\cdot\na\e e^{-u}\\
\le&\int_Mu^2|\na\e|^2e^{-u}\le\f1{r^2}\int_Mu^2e^{-u}.
\endaligned
\end{equation}
Here the volume form $d\mu$ is omitted in the above integrations for notational simplicity. Passing to the limit as $r\rightarrow+\infty$ we get $u\equiv0$, but this is not a solution to \eqref{tsl}.
\end{proof}
We choose a local orthonormal frame field $\{e_1,\cdots, e_n\}$ of $M$ and let $\overline{\na}$ be the Levi-Civita connection of $\R^{n+1}_1$
as before. Let $B$ be the second fundamental form and $B_{e_i e_j}=h(e_i,e_j)\nu=h_{ij}\nu$. Then the coefficients of the second fundamental form $h_{ij}$ is a symmetric $2-$tensor on $M$ and
\begin{equation}\aligned\label{hij}
h_{ij}=-\lan\overline{\na}_{e_i}e_j,\nu\ran.
\endaligned
\end{equation}
By \eqref{HnuE}, mean curvature $H=\sum_ih_{ii}=\f1{\sqrt{1-|Du|^2}}$. Denote the square of the second fundamental form $|B|^2=\sum_{i,j}h_{ij}^2$.
\begin{lemma}\label{LhijB}
Let $h_{ij}$ and $L$ be as defined previously. In the meaning of covariant, we have
\begin{equation}\aligned\label{Lh}
Lh_{ij}=\De h_{ij}+\lan E_{n+1},\na h_{ij}\ran=|B|^2h_{ij}.
\endaligned
\end{equation}
\end{lemma}
\begin{proof}
By Ricci identity, one has
\begin{equation}\aligned
\De h_{ij}=h_{ijkk}=h_{ikjk}=h_{ikkj}+h_{il}R_{klkj}+h_{kl}R_{ilkj}.
\endaligned
\end{equation}
Combining Gauss formula $R_{ijkl}=-h_{ik}h_{jl}+h_{il}h_{jk}$(see \cite{CY} or \cite{O} for example) and \eqref{HnuE}, we have
\begin{equation}\aligned\label{Deh}
\De h_{ij}=&H_{ij}+h_{il}(-h_{kk}h_{lj}+h_{kj}h_{kl})+h_{kl}(-h_{ik}h_{lj}+h_{ij}h_{kl})\\
=&-(\lan E_{n+1},\nu\ran)_{ij}-Hh_{ik}h_{jk}+|B|^2h_{ij}.
\endaligned
\end{equation}
Since \begin{equation}\aligned
\na_{e_i}(h(e_j,e_k))=&(\na_{e_i}h)(e_j,e_k)+h(\na_{e_i}e_j,e_k)+h(e_j,\na_{e_i}e_k)\\
=&h_{jki}+\lan\na_{e_i}e_j,e_l\ran h_{kl}+\lan\na_{e_i}e_k,e_l\ran h_{jl},
\endaligned
\end{equation}
then
\begin{equation}\aligned\label{dH}
-(\lan E_{n+1},\nu\ran)_{ij}=&-\na_{e_i}\na_{e_j}\lan E_{n+1},\nu\ran+\na_{\na_{e_i}{e_j}}\lan E_{n+1},\nu\ran\\
=&-\na_{e_i}(\lan E_{n+1},e_k\ran h_{jk})+\lan\na_{e_i}e_j,e_k\ran\lan E_{n+1},\overline{\na}_{e_k}\nu\ran\\
=&-\lan E_{n+1},\overline{\na}_{e_i}e_k\ran h_{jk}-\lan E_{n+1},e_k\ran\big(h_{jki}+\lan\na_{e_i}e_j,e_l\ran h_{kl}\\&+\lan\na_{e_i}e_k,e_l\ran h_{jl}\big)+\lan\na_{e_i}e_j,e_k\ran\lan E_{n+1},{e_l}\ran h_{kl}\\
=&-\lan E_{n+1},\nu\ran h_{ik}h_{jk}-\lan E_{n+1},\na_{e_i}e_k\ran h_{jk}-\lan E_{n+1},e_k\ran h_{ijk}\\
&-\lan E_{n+1},e_k\ran\lan\na_{e_i}e_k,e_l\ran h_{jl}\\
=&-\lan E_{n+1},\nu\ran h_{ik}h_{jk}-\lan E_{n+1},e_l\ran\lan e_l,\na_{e_i}e_k\ran h_{jk}-\lan E_{n+1},\na h_{ij}\ran\\
&+\lan E_{n+1},e_k\ran\lan e_k,\na_{e_i}e_l\ran h_{jl}\\
=&H h_{ik}h_{jk}-\lan E_{n+1},\na h_{ij}\ran.
\endaligned
\end{equation}
Combining \eqref{Lf}, \eqref{Deh} and \eqref{dH}, we complete the Lemma.
\end{proof}

Let $u$ be a smooth solution to \eqref{tsl} in $\R^n$. For any constant $h>\inf_{x\in\R^n}u(x)$, we denote
\begin{equation}\aligned
\G_{h,u}=&\{x\in\R^n|\ u(x)=h\},\\
\Om_{h,u}=&\{x\in\R^n|\ u(x)<h\}.
\endaligned
\end{equation}
and call them the level set and the sublevel set of $u$, respectively.

Now we restrict $\G_{h,u}$ as a hypersurface in $\R^n$. Let $\g$ be the unit outward normal vector of $\G_{h,u}$ and $\De_\G$ be the Laplacian operator of $\G_{h,u}$. Let $\na^\G$ and $\na^{\R^n}$ be the Levi-Civita connections of $\G_{h,u}$ and $\R^n$, respectively. If $\{\be_i\}_{i=1}^{n-1}$ is an orthonormal frame of $\G_{h,u}$, then the mean curvature of $\G_{h,u}$  $$H_\G\triangleq-\sum_i\lan\na^{\R^n}_{\be_i}\be_i,\g\ran.$$
At any point of $\G_{h,u}$,
\begin{equation}\aligned\label{LRu}
\De_{\R^n} u=&\sum_i\big(\be_i\be_iu-(\na^{\R^n}_{\be_i}\be_i)u\big)+u_{\g\g}=\sum_i\big(\be_i\be_iu-(\na^{\G}_{\be_i}\be_i)u\big)+H_\G u_\g+u_{\g\g}\\
=&\De_\G u+H_\G u_\g+u_{\g\g}=H_\G u_\g+u_{\g\g},
\endaligned
\end{equation}
and
\begin{equation}\aligned\label{uijij}
u_iu_ju_{ij}=\f12u_i\p_i(|Du|^2)=\f12Du\cdot Du_\g^2=\f12\p_\g u\cdot \p_\g u_\g^2=u_\g^2u_{\g\g}.
\endaligned
\end{equation}
Then combining \eqref{ts}\eqref{LRu}\eqref{uijij}, we have (compared to the Euclidean case\cite{W1})
\begin{equation}\aligned\label{ug}
H_\G u_\g+\f{u_{\g\g}}{1-u_{\g}^2}=1.
\endaligned
\end{equation}

\begin{lemma}\label{Conv}
If $\G_{h,u}$ is a nonempty convex compact set in $\R^n$ with mean curvature $H_\G\le1$, then $u$ is convex in $\Om_{h,u}$.
\end{lemma}
\begin{proof}
Since $\G_{h,u}$ is convex, then $0\le H_\G\le1$. Combining $|u_\g|\le1$ and \eqref{ug}, we have
\begin{equation}\aligned\label{ugg}
u_{\g\g}=(1-u_{\g}^2)(1-H_\G u_\g)\ge0,
\endaligned
\end{equation}
which implies $D^2u\big|_{\G_{h,u}}\ge0$.
Let $\f{\p}{\p x_i}=E_i+u_iE_{n+1}$ be a tangential vector field in $M$, then
\begin{equation}\aligned\label{hxixj}
h\left(\f{\p}{\p x_i},\f{\p}{\p x_j}\right)=-\left\lan\overline{\na}_{\f{\p}{\p x_i}}\f{\p}{\p x_j},\nu\right\ran=-u_{ij}\lan E_{n+1},\nu\ran=\f{u_{ij}}{\sqrt{1-|Du|^2}}.
\endaligned
\end{equation}
Let $\la(x)$ be the minimal principal curvature of the second fundamental form at the point $x\in M$(see \cite{CY} or \cite{T} for Ricci curvature). Then \eqref{hxixj} implies $\la\big|_{\G_{h,u}}\ge0$.
If $\la(x)$ is not nonnegative in $\Om_{h,u}$, then there is a point $p_0\in\Om_{h,u}$ such that $\la(p_0)=\inf_{x\in\Om_{h,u}}\la(x)<0$. Choose a local orthonormal frame $\{e_i\}$ near $p_0$ in $M$ and denote $h_{ij}=h(e_i,e_j)$ as mentioned before. Let $\th=\sum_i\th_ie_i\big|_{p_0}$ be a unit eigenvector of the second fundamental form with eigenvalue $\la(p_0)$ at the point $p_0$, namely, $h(\th,\th)=\la(p_0)$. Then we define a smooth function by
$$f(x)=\sum_{i,j}h_{ij}\Big|_x\th_i\th_j.$$
$f$ attains the minimal value $\la(p_0)$ in a neighborhood of $p_0$.
At the point $p_0$, by \eqref{Lh} we have
\begin{equation}\aligned
0\le Lf= L(h_{ij}\th_i\th_j)=L(h_{ij})\th_i\th_j=|B|^2h_{ij}\th_i\th_j=|B|^2f\le\f{H^2}nf<0.
\endaligned
\end{equation}
This is a contradiction. Therefore, $(h_{ij})$ is nonnegative in $\overline{\Om}_{h,u}$, which yields the Lemma.
\end{proof}
\begin{lemma}\label{tds}
Let $\Om$ be a bounded convex domain with smooth boundary in $\R^n$ and $\Om_\si=\{\sigma x|\ x\in\Om\}$ for $\si\in(0,1]$. If $u_\si$ is a smooth strictly spacelike solution to \eqref{ts} in $\Om_\si$ with $u_\si\big|_{\p\Om_\si}=0$, then there is a constant $\th\in(0,1)$ depending only on the diameter of $\overline{\Om}$, such that $\max_{\overline{\Om}_\si}|Du_\si|\le1-\th$.
\end{lemma}
\begin{proof}
Let $d=\mathrm{diam}(\overline{\Om})$ be the diameter of $\overline{\Om}$. For any fixed $y\in\p\Om_\si$, we assume $y=(y_1,0\cdots,0)$ and $\Om_\si\subset\{x=(x_1,\cdots,x_n)|\ -y_1<x_1<y_1\}$ with $0<y_1\le\f d2$. Let $\varphi(x)=\log\cosh(x_1)-\log\cosh(y_1)$, then $\mathcal{L}\varphi=1$, $\varphi(y)=0$ and $\varphi\big|_{\p\Om_\si}\le0$. Using Lemma \ref{max} for $\pm\varphi$, one gets
$$\varphi(x)\le u_\si(x)\le-\varphi(x) \qquad \mathrm{for}\ x\in\Om_\si.$$
Combining $u_\si\big|_{\p\Om_\si}=\varphi(y)=0$ gives
$$|Du_\si(y)|\le|D\varphi(y)|\le\tanh d.$$
By the maximum principle for \eqref{bge}, we complete the proof.
\end{proof}

Now let us consider a Dirichlet problem, which may be inconvenient to be found out directly in general theories of PDE.
\begin{theorem}\label{Dirichlet}
Let $\Om$ be a bounded convex domain with smooth boundary in $\R^n$. Then the Dirichlet problem $\mathcal{L}u=1$ in $\Om$ with $u\big|_{\p\Om}=0$ has a smooth strictly spacelike solution $u$ on $\overline{\Om}$.
\end{theorem}
\begin{proof}
For $f\in C^{2,\a}(\overline{\Om})$, $|Df|<1$ and $\sigma\in(0,1]$, we define $Q_\sigma$ by
\begin{equation}\aligned
Q_\sigma f=\sum_{i,j}\left(\de_{ij}+\f{f_if_j}{1-|Df|^2}\right)f_{ij}-\sigma.
\endaligned
\end{equation}
If $u_\sigma\in C^{2,\a}(\overline{\Om})$ is a strictly spacelike solution to $Q_\sigma u_\sigma=0$ in $\Om$ and $u_\sigma\big|_{\p\Om}=0$, then $u_\si\in C^{\infty}(\overline{\Om})$ by standard regularity theory of elliptic equations. Let $\Om_\sigma=\{\sigma x|\ x\in\Om\}$ and
\begin{equation}\aligned\label{uusi}
u(x)=\si u_\si\left(\f x{\si}\right),
\endaligned
\end{equation}
then $u$ is a smooth strictly spacelike solution to \eqref{ts} in $\Om_\sigma$.

By Lemma \ref{tds}, there is a constant $\th\in(0,1)$ depending only on the diameter of $\overline{\Om}$, such that $\max_{\overline{\Om}_\si}|Du|\le1-\th$. If the mean curvature of $\p\Om$ satisfies $0\le H_{\p\Om}\le C_1$ for some $C_1\ge1$, then $$0\le H_{\p\Om_\si}\le \f{C_1}{\si}.$$ Let $\g$ be the unit outward normal vector of $\p\Om_\si$, then by \eqref{ugg}, we have $1-\f{C_1}{\si}\le u_{\g\g}\le1$. Hence, $$D^2u+\left(\f{C_1}{\si}-1\right)I$$ is positive definite on $\p\Om_\si$, where $I$ is a unit $n\times n$ matrix. Moreover, $1-\f{C_1}{\si}\le u_{ii}\le C_2$ for some $C_2$ depending only on $n,\Om$ and
$$\big(u_{ij}\big)^2\le \left(u_{ii}+\f{C_1}{\si}-1\right)\left(u_{jj}+\f{C_1}{\si}-1\right)\le\left(C_2+\f{C_1}{\si}-1\right)^2\qquad  \mathrm{for}\ i\neq j.$$
Let $h_{ij}$ be defined as \eqref{hij}, where we replace $M$ by $\{(x,u(x))|\ x\in\Om_\si\}$. By the proof of Lemma \ref{Conv}, the matrix $(h_{ij})$ must attain its negative minimal eigenvalue and positive maximal eigenvalue on the boundary $\p\Om_\si$. Combining \eqref{hxixj} and $\max_{\overline{\Om}_\si}|Du|\le1-\th$, we get that there is a constant $C_3$ depending only on $n,\Om$ such that $$\max_{x\in\overline{\Om}_\si}u_{ij}(x)\le\f{C_3}{\si} \qquad\qquad \mathrm{for}\ i,j=1,\cdots,n.$$
\eqref{uusi} implies $\max_{\overline{\Om}}\p_{ij}u_{\si}\le C_3$. For some $\be\in(0,1)$ there is a constant $C_\be$ depending only on $n,\be,\Om$ such that
\begin{equation}\aligned\label{u1si}
|u_\si|_{C^{1,\be}(\overline{\Om})}\le C_\be\quad \mathrm{and}\quad \max_{\overline{\Om}}|Du_\si|\le1-\th.
\endaligned
\end{equation}

For any $f\in C^{1,\be}(\overline{\Om})$ with $\max_{\overline{\Om}}|Df|<1$, the operator $T$ is defined by letting $u=Tf$ be the unique solution in $C^{2,\a\be}(\overline{\Om})$ of the linear Dirichlet problem (see \cite{GT} for example):
$$\sum_{i,j}\left(\de_{ij}+\f{f_if_j}{1-|Df|^2}\right)u_{ij}-1=0\ \ \ \mathrm{in}\ \Om,\ \ \ u\big|_{\p\Om}=0.$$
If $|\cdot|_{C^{1,\be}(\overline{\Om})}$ is a norm in H$\mathrm{\ddot{o}}$lder space $C^{1,\be}(\overline{\Om})$, then we may define a new norm $\|\cdot\|$ in $C^{1,\be}(\overline{\Om})$ by $$\|f\|=\max_{\overline{\Om}}|Df|+\f{\th}{2C_\be}|f|_{C^{1,\be}(\overline{\Om})}.$$
Let $\mathfrak{S}=\{f\in C^{1,\be}(\overline{\Om})\big|\ \|f\|\le1-\f{\th}4\}$ be a closed ball in $C^{1,\be}(\overline{\Om})$. We define a mapping $T^*$ by
\begin{eqnarray*}
   T^*f\triangleq \left\{\begin{array}{cc}
     Tf     & \quad\ \ \ {\rm{if}} \ \ \  \|Tf\|\le 1-\f{\th}4 \\ [3mm]
     \left(1-\f{\th}4\right)\f{Tf}{\|Tf\|}  & \quad\quad\  {\rm{if}} \ \ \  \|Tf\|\ge 1-\f{\th}4,
     \end{array}\right.
\end{eqnarray*}
then $T^*$ is a mapping from $\mathfrak{S}$ to $\mathfrak{S}$. By Arzela's theorem, $T^*$ is a compact operator in $\mathfrak{S}$. By Corollary 11.2 in \cite{GT}, $T^*$ has a fixed point $w$ in $\mathfrak{S}$. If $\|Tw\|\ge1-\f{\th}4$, then $w=T^*w=\si Tw$ if $\si=\left(1-\f{\th}4\right)\f1{\|Tw\|}$, and $\|w\|=\|T^*w\|=1-\f{\th}4$. The definition of $T$ implies
$$\sum_{i,j}\left(\de_{ij}+\f{w_iw_j}{1-|Dw|^2}\right)w_{ij}-\si=0\ \ \ \mathrm{in}\ \Om\ \ \mathrm{and}\ \ \ w\big|_{\p\Om}=0.$$
By \eqref{u1si}, we have $\|w\|\le1-\th+\f{\th}{2C_\be}C_\be=1-\f{\th}2,$ which is a contradiction. Hence, $\|Tw\|<1-\f{\th}4$ and $w=Tw$. By the standard regularity theory of elliptic equations, $w$ is our desired function.
\end{proof}

\section{Entire spacelike translating solitons}

By \cite{E}, the elliptic equation \eqref{tsl} has a radially symmetric solution $\psi(r)$, where $\psi(r)\in C^2([0,+\infty))$ satisfies the following ODE
\begin{equation}\aligned\label{rts}
\f{\psi''}{1-(\psi')^2}+\f{n-1}r\psi'=1\qquad \mathrm{for}\ r\in (0,+\infty)
\endaligned
\end{equation}
with $\psi'(0)=0$ in an appropriate coordinate system. By \cite{J}, up to an additive constant \eqref{rts} has a unique smooth convex solution $\psi$ with $\f{r}{\sqrt{n^2+r^2}}\le\psi'<1$ for $r\ge0$. Thus, for $r\ge0$ one has
\begin{equation}\aligned\label{psir}
r-n\le\int_0^r\f{s}{\sqrt{n^2+s^2}}ds\le\int_0^r\psi'(s)ds=\psi(r)-\psi(0)\le r.
\endaligned
\end{equation}

\begin{lemma}\label{Converge}
Let $\Om_k$ be a family of convex domains with smooth boundaries ($\p\Om_k$ may be empty), $\overline{\Om}_k\subset\Om_{k+1}$ and $\bigcup_{k\ge1}\Om_k=\R^n$. Let $u_k$ be a smooth convex strictly spacelike function to \eqref{ts} in $\Om_k$. Denote $x_t=x+t(y-x)$ for $x,y\in\R^n$. Suppose that there is a function $W$ in $\R^n$ satisfying
$$\limsup_{t\rightarrow+\infty}\f{|W(x_t)-W(x_{-t})|}{|x_t-x_{-t}|}<1$$
for any $x\neq y\in\R^n$. If $\limsup_{k\rightarrow\infty}|u_k(x)-W(x)|\le C$ for some absolute constant $C$ and any $x\in\R^n$, then there is a subsequence of $\{u_k\}$ converging to an entire smooth convex strictly spacelike function $u$ to \eqref{ts} in $\R^n$.
\end{lemma}
\begin{proof}
By convexity of $u_k$ and \eqref{ts}, we have
$$1=\sum_{i=1}^n\left(\de_{ij}+\f{\p_iu_k\p_ju_k}{1-|Du_k|^2}\right)\p_{ij}u_k\ge\sum_i\p_{ii}u_k.$$
For any $i,j=1,\cdots,n$ one gets $$\big(\p_{ij}u_k\big)^2\le\p_{ii}u_k\p_{jj}u_k\le\left(\f{\p_{ii}u_k+\p_{jj}u_k}2\right)^2\le1.$$
There is a subsequence $\{u_{k_j}\}$ of $\{u_j\}$ converging to $u$ in $C^1(K)$ uniformly for any compact set $K\subset\R^n$. Then $u\in C^{1,1}(\R^n)$ is a convex function with $|Du|\le1$. Clearly,
$$W(x)-C\le u(x)\le W(x)+C\qquad \mathrm{for\ each\ } \ x\in\R^n.$$
Let $\Om$ be an arbitrary bounded domain in $\R^n$. By the definition of $F_{\Om}(\cdot)$ in $\S$ 2 and the Appendix, there is a weakly spacelike function $w_0\in \mathcal{C}(u,\Om)$ such that $$F_{\Om}(w_0)=\sup_{w\in\mathcal{C}(u,\Om)}F_{\Om}(w).$$
For any $\ep>0$, denote $\Om_{j,\ep}^+=\{x\in\Om|\ u_{k_j}(x)>w_0(x)-\ep\}$ and $\Om_\ep^+=\{x\in\Om|\ u(x)>w_0(x)-\ep\}$.
By Theorem \ref{uni} one has
$$F_{\Om_{j,\ep}^+}(w_0-\ep)\le F_{\Om_{j,\ep}^+}(u_{k_j})\le F_{\Om_\ep^+}(u_{k_j})+F_{\Om_{j,\ep}^+\setminus\Om_\ep^+}(u_{k_j}).$$
If the set $\G_\ep=\{x\in\Om|\ u(x)=w_0(x)-\ep\}$ has positive $n$-dimensional Lebesgue measure, then for any open set $U_\ep\subset\G_\ep$, $F_{U_\ep}(w_0-\ep)=F_{U_\ep}(u)$.
Since $u_{k_j}\rightarrow u$ in $C^1-$norm, then $$\Om^+_\ep\subset\liminf_{j\to\infty}\Om_{j,\ep}^+\subset\limsup_{j\to\infty}\Om_{j,\ep}^+\subset\Om_\ep^+\cup\G_\ep.$$
Letting $j\to\infty$ gives
\begin{equation}\label{eepFomep}
e^{\ep}F_{\Om_\ep^+}(w_0)=F_{\Om_\ep^+}(w_0-\ep)\le F_{\Om_\ep^+}(u).
\end{equation}
If the set $\G=\{x\in\Om|\ u(x)=w_0(x)\}$ has positive $n$-dimensional Lebesgue measure, then for any open set $U\subset\G$, $F_{U}(w_0)=F_{U}(u)$. Denote $\Om^+=\{x\in\Om|\ u(x)>w_0(x)\}$ and $\Om^-=\{x\in\Om|\ u(x)<w_0(x)\}$, then $$\Om^+\subset\liminf_{\ep\to0}\Om_{\ep}^+\subset\limsup_{\ep\to0}\Om_{\ep}^+\subset\Om^+\cup\G.$$
Let $\ep\to 0$ in \eqref{eepFomep}, we obtain
$$F_{\Om^+}(w_0)\le F_{\Om^+}(u).$$
By the same way, one gets $F_{\Om^-}(w_0)\le F_{\Om^-}(u).$ $\Om=\Om^+\cup\Om^-\cup\G$ implies
$$F_\Om(w_0)=F_{\Om^+}(w_0)+F_{\Om^-}(w_0)+F_{\G}(w_0)\le F_{\Om}(u).$$
Thus the definition of $w_0$ tells us that $u$ is a weak solution to the variational functional $F$ in any bounded domain $\Om$.

If there is $x\neq y\in\R^n$ such that $|u(x)-u(y)|=|x-y|$. Then by Lemma \ref{line}, we obtain
\begin{equation}\aligned\label{uxt}
|u(x_t)-u(x_{-t})|=|x_t-x_{-t}| \qquad \mathrm{for}\ \forall t\in\R,
\endaligned
\end{equation}
where $x_t=x+t(y-x)$. The definition of $W$ and $W(x)-C\le u(x)\le W(x)+C$ imply that \eqref{uxt} is impossible. Hence $|u(x)-u(y)|<|x-y|$ for any $x\neq y$. Noting the convexity of $u\in C^{1,1}(\R^n)$, we get $|Du|<1$.
The regularity theory of elliptic equations and $u\in C^{1,1}(\R^n)$ force that $u$ is a smooth solution to \eqref{tsl}.
\end{proof}
Let $Q$ be defined as in $\S$ 1 and $Q_0$ be the set that contains all linear functions through the origin whose gradient has norm one.
Denote $\Sp^{n-1}$ be the unit sphere centered at the origin, then for any $V\in Q\setminus Q_0$ there is a closed set $\Lambda\subset\Sp^{n-1}$
containing two points at least, such that $V(x)=\sup_{\la\in\Lambda}\lan\la,x\ran$ (see \cite{CT} for example).
\begin{theorem}
For $n\ge2$ and any $V\in Q\setminus Q_0$ there is an entire smooth convex strictly spacelike solution $u$ to \eqref{tsl} such that $u$ blows down to $V$.
\end{theorem}
\begin{proof} (of Theorem \ref{Vcu}).
For any fixed positive constant $K$ and $V\in Q\setminus Q_0$, let $\widetilde{V}(x)\triangleq\max\{V(x),|x|-K\}$, then
$$\lim_{|x|\rightarrow\infty}\f{\widetilde{V}(x)}{|x|}=1.$$
Clearly, $\widetilde{V}(x)$ is a convex weakly spacelike function in $\R^n$ with $\widetilde{V}(0)=0$. Let $\widetilde{V}_\ep$ be a mollifier of $\widetilde{V}$ defined by
\begin{equation}\aligned
\widetilde{V}_\ep(x)=\int_{\R^n}\r(y)\widetilde{V}(x-\ep y)dy=\f1{\ep^n}\int_{\R^n}\r\left(\f{x-y}{\ep}\right)\widetilde{V}(y)dy,
\endaligned
\end{equation}
where $\ep$ is a positive constant to be determined below and $\r$ is the function defined in $\S$ 2. Then
\begin{equation}\aligned\label{Vtep}
\big|\widetilde{V}_\ep(x)-\widetilde{V}(x)\big|=\left|\int_{\R^n}\r(y)\big(\widetilde{V}(x-\ep y)-\widetilde{V}(x)\big)dy\right|\le\int_{\R^n}\r(y)\ep|y|dy\le\ep,
\endaligned
\end{equation}
and $|D\widetilde{V}_\ep(x)|\le1$ by \eqref{Dw}. Denote $\Om_{h,\widetilde{V}_\ep}=\{x\in\R^n|\ \widetilde{V}_\ep(x)<h\}$ and $\G_{h,\widetilde{V}_\ep}=\p\Om_{h,\widetilde{V}_\ep}$. The convexity of $\widetilde{V}_\ep$ implies that $\Om_{h,\widetilde{V}_\ep}$ is convex. Let $\{\xi_1,\cdots,\xi_n\}$ be an orthonormal coordinate transform of $\{x_1,\cdots,x_n\}$ such that $\f{\p}{\p\xi_n}$ is the outward normal vector of $\G_{h,\widetilde{V}_\ep}$. For sufficiently large $h>0$ and each $x\in\G_{h,\widetilde{V}_\ep}$,
\begin{equation}\aligned\label{xinV1}
\p_{\xi_n}\widetilde{V}_\ep=|D\widetilde{V}_\ep|\ge\f{\p}{\p r}\widetilde{V}_\ep\ge\f{\widetilde{V}_\ep(x)-\widetilde{V}_\ep(0)}{|x|}\ge\f{\widetilde{V}(x)-2\ep}{|x|}\ge\f12\ ,
\endaligned
\end{equation}
where we have used the convexity of $\widetilde{V}_\ep$ in the second inequality of \eqref{xinV1}.

Since
\begin{equation}\aligned
\p_{ij}\widetilde{V}_\ep(x)=&\f1{\ep^{n}}\int_{\R^n}\p_{x_ix_j}\left(\r\left(\f{x-y}{\ep}\right)\right)\widetilde{V}(y)dy
=-\f1{\ep^n}\int_{\R^n}\p_{y_j}\p_{x_i}\left(\r\left(\f{x-y}{\ep}\right)\right)\widetilde{V}(y)dy\\
=&\f1{\ep^n}\int_{\R^n}\p_{x_i}\left(\r\left(\f{x-y}{\ep}\right)\right)\p_{y_j}\widetilde{V}(y)dy,
\endaligned
\end{equation}
then one has
\begin{equation}\aligned
|\p_{ij}\widetilde{V}_\ep(x)|\le\f1{\ep^n}\int_{\R^n}\left|\p_{x_i}\left(\r\left(\f{x-y}{\ep}\right)\right)\right|dy
\le\f1{\ep}\int_{B_1}\big|D\r(x)\big|dx=\f{C_4}{\ep}.
\endaligned
\end{equation}
Here, $C_4$ is a constant depending only on $n$.

For any fixed $y\in\G_{h,\widetilde{V}_\ep}$, selecting $\{\xi_1,\cdots,\xi_{n-1}\}$ such that $$(\p_{\xi_i\xi_j}\xi_n)_{(n-1)\times(n-1)}\big|_y=\mathrm{diag}\{\la_1,\cdots,\la_{n-1}\}.$$
Since $\Om_{h,\widetilde{V}_\ep}$ is convex, then $\la_i\le0$ and $\p_{\xi_i}\xi_n\big|_y=0$ for $i=1,\cdots,n-1$. Taking derivations of the equation $\widetilde{V}_\ep(x)=h$ at $y$ gives
\begin{equation}\aligned\label{d1V1}
\p_{\xi_i}\widetilde{V}_\ep+\p_{\xi_n}\widetilde{V}_\ep\p_{\xi_i}\xi_n=0,
\endaligned
\end{equation}
and
\begin{equation}\aligned\label{d2V1}
\p_{\xi_i\xi_j}\widetilde{V}_\ep+\p_{\xi_i\xi_n}\widetilde{V}_\ep\p_j\xi_n+\p_{\xi_j\xi_n}\widetilde{V}_\ep\p_{\xi_i}\xi_n+\p_{\xi_n}
\widetilde{V}_\ep\p_{\xi_i\xi_j}\xi_n+\p_{\xi_n\xi_n}\widetilde{V}_\ep\p_{\xi_i}\xi_n\p_{\xi_j}\xi_n=0.
\endaligned
\end{equation}
From $\p_{\xi_i}\xi_n\big|_y=0$, one gets
\begin{equation}\aligned\label{V1ij}
\p_{\xi_i\xi_j}\widetilde{V}_\ep=-\p_{\xi_n}\widetilde{V}_\ep\p_{\xi_i\xi_j}\xi_n=-\p_{\xi_n}\widetilde{V}_\ep\la_i\de_{ij}.
\endaligned
\end{equation}
The convexity of $\widetilde{V}_\ep$ implies
$$\sum_{i=1}^{n-1}\p_{\xi_i\xi_i}\widetilde{V}_\ep\le\sum_{i=1}^{n}\p_{\xi_i\xi_i}\widetilde{V}_\ep
=\sum_{i=1}^{n}\p_{ii}\widetilde{V}_\ep\le\f{nC_4}\ep.$$
Then for sufficiently large $h$ the mean curvature of $\G_{h,\widetilde{V}_\ep}$ satisfies
\begin{equation}\aligned
0\le H_{\G_{h,\widetilde{V}_\ep}}=&-\f{1}{\sqrt{1+\sum_l(\p_{\xi_l}\xi_n)^2}}\sum_{i,j=1}^{n-1}\big(\de_{ij}
-\f{\p_{\xi_i}\xi_n\p_{\xi_j}\xi_n}{1+\sum_l(\p_{\xi_l}\xi_n)^2}\big)\p_{\xi_i\xi_j}\xi_n\\
=&-\sum_{i=1}^{n-1}\p_{\xi_i\xi_i}\xi_n
=-\sum_i\la_i=\sum_{i=1}^{n-1}\f{\p_{\xi_i\xi_i}\widetilde{V}_\ep}{\p_{\xi_n}\widetilde{V}_\ep}\le\f{nC_4}{\ep\p_{\xi_n}\widetilde{V}_\ep}
\le\f{2nC_4}{\ep}.
\endaligned
\end{equation}
Let $\ep=2nC_4$, and denote $\Om_k=\Om_{k,\widetilde{V}_\ep}$ for all sufficiently large $k$. Then we get a family of convex domains $\{\Om_k\}$ with smooth boundaries and
$0\le H_{\p\Om_k}\le1$.

By Theorem \ref{Dirichlet}, there is a smooth solution $u_k$ to \eqref{ts} in $\overline{\Om}_k$ with $u_k\big|_{\p \Om_k}=k$ and $|Du_k|<1$ in $\overline{\Om}_k$. Since $|H_{\p\Om_k}|\le1$, then by Lemma \ref{Conv}, $u_k$ is a smooth convex strictly spacelike function in $\Om_k$. By \eqref{Vtep}, $|\widetilde{V}(x)-\widetilde{V}_\ep(x)|\le \ep$, then for any $x\in\p\Om_k$, we have
\begin{equation}\aligned\label{Vuk}
\widetilde{V}(x)-\ep\le u_k(x)=k\le \widetilde{V}(x)+\ep.
\endaligned
\end{equation}
Let $\psi$ be the solution of \eqref{rts} with $\psi'(0)=0$ and $\psi(0)=0$. Combining \eqref{psir} and $|D\widetilde{V}|\le1$ a.e. gives
\begin{equation}\aligned\label{Vxy}
\widetilde{V}(x)-\widetilde{V}(y)\le|x-y|\le\psi(|x-y|)-\psi(0)+n=\psi(|x-y|)+n.
\endaligned
\end{equation}
Combining \eqref{Vuk} and \eqref{Vxy} implies $u_k(x)\le \widetilde{V}(x)+\ep\le \widetilde{V}(y)+\psi(|x-y|)+n+\ep$ for $x\in\p\Om_k$. By Lemma \ref{comp}, we have $u_k(x)\le \widetilde{V}(y)+\psi(|x-y|)+n+\ep$ for $x\in \Om_k$ and each fixed $y\in\R^n$. Hence,
\begin{equation}\aligned\label{up}
u_k(x)\le \inf_{y\in\R^n}\{\widetilde{V}(y)+\psi(|x-y|)\}+n+\ep\le \widetilde{V}(x)+n+\ep \qquad \mathrm{for}\ x\in \Om_k.
\endaligned
\end{equation}
There is a closed set $\Lambda\subset\Sp^{n-1}$ containing two points at least, such that $V(x)=\sup_{\la\in\Lambda}\lan\la,x\ran$. Then \eqref{Vuk} implies $\lan\la,x\ran-\ep\le u_k(x)$ on $\p \Om_k$ for every $\la\in\Lambda$. Applying Lemma \ref{comp} we get $\lan\la,x\ran-\ep\le u_k(x)$ in $\Om_k$. Similarly, for each $\widetilde{\la}\in\Sp^{n-1}$, $\lan\widetilde{\la},x\ran-K-\ep\le u_k(x)$ in $\Om_k$. Hence $\widetilde{V}(x)-\ep\le u_k(x)$ in $\Om_k$ by $|x|=\sup_{\tilde{\la}\in\Sp^{n-1}}\lan\tilde{\la},x\ran$. Combining \eqref{up}, we obtain
\begin{equation}\aligned\label{con}
\widetilde{V}(x)-\ep\le u_k(x)\le \widetilde{V}(x)+n+\ep \qquad \mathrm{for}\ x\in \Om_k.
\endaligned
\end{equation}

There is a subsequence of $\{u_k\}$ converging to a convex weakly spacelike function $u_K$ uniformly in any compact set in $\R^n$. Clearly, $\widetilde{V}(x)-\ep\le u_K(x)\le \widetilde{V}(x)+n+\ep$ for $x\in\R^n$. By Lemma \ref{Converge} and the definition of $\widetilde{V}$, we know that $u_K$ is an entire smooth convex strictly spacelike solution to \eqref{tsl}.

Note that $\widetilde{V}(x)=\max\{V(x),|x|-K\}$ depends on $K$. Let $\{K_i\}$ be a sequence in $\R^+$ satisfying $\lim_{i\rightarrow\infty}K_i=+\infty$. Let $u_{K_i}$ be an entire smooth convex strictly spacelike function to \eqref{tsl} with $$\max\{V(x),|x|-K_i\}-\ep\le u_{K_i}(x)\le \max\{V(x),|x|-K_i\}+n+\ep\ \qquad \mathrm{in}\ \R^n.$$ By Lemma \ref{Converge} and the definition of $V$, there is a subsequence of $\{u_{K_i}\}$ converging to an entire smooth convex strictly spacelike function $u$ satisfying \eqref{tsl} in $\R^n$. Moreover, $V(x)-\ep\le u(x)\le V(x)+n+\ep$ for $x\in\R^n$, which implies
$$\lim_{r\rightarrow\infty}\f{u(rx)}r=V(x).$$
Hence $u$ blows down to $V$.
\end{proof}

By Theorem 1.2 in \cite{J}, if $u$ is an entire smooth convex strictly spacelike solution to \eqref{tsl}, then $u$ must blow down to some function $V$ in $Q$. Namely, $V$ is a convex homogeneous of degree one function such that $|DV(y)|=1$ for any $y\in\R^n$ where the gradient exists.

Now, let us explore the extent of non-uniqueness for spacelike solutions for given projective data at infinity. Compared to Theorem 2 in \cite{T}, we obtain a strong result for translating solitons.
\begin{theorem}
Let $n\ge2$ and $f\in C^2(\Sp^{n-1})$. Then there exists a unique entire smooth convex strictly spacelike solution $u$ to \eqref{tsl} such that
$$\lim_{|x|\rightarrow\infty}\left(u(x)-|x|-f\big(\f x{|x|}\big)\right)=0.$$
\end{theorem}
\begin{proof}
By Theorem 2 in \cite{T}, there are $C^1$ functions $p_1,p_2$ on $\Sp^{n-1}$ such that
$$\lan p_1(\e),\xi-\e\ran\le f(\xi)-f(\e)\le\lan p_2(\e),\xi-\e\ran\qquad \mathrm{for}\ \xi,\e\in\Sp^{n-1}.$$
Let $\psi$ be the smooth function satisfying \eqref{rts} with $\psi'(0)=0$ and $\lim_{r\rightarrow\infty}\big(\psi(r)-r\big)=0$. Let
$z_i(x,\e)=f(\e)-\lan p_i(\e),\e\ran+\psi(|x+p_i(\e)|)$ for $x\in\R^n$, $\e\in\Sp^{n-1}$ and $i=1,2$. By $\lim_{r\rightarrow\infty}\big(\psi(|r\xi+p_i(\e)|)-r\big)=\lan p_i(\e),\xi\ran$ uniformly for $\xi,\e\in\Sp^{n-1}$, we have
\begin{equation}\aligned
\lim_{r\rightarrow\infty}\big(z_1(r\xi,\e)-r\big)=&f(\e)-\lan p_1(\e),\e\ran+\lim_{r\rightarrow\infty}\big(\psi(|r\xi+p_1(\e)|)-r\big)\\
=&f(\e)-\lan p_1(\e),\e\ran+\lan p_1(\e),\xi\ran\le f(\xi).
\endaligned
\end{equation}
Similarly, we have $f(\xi)\le\lim_{r\rightarrow\infty}\big(z_2(r\xi,\e)-r\big)$, where equality holds when $\xi=\e$. Let $q_1(x)=\sup_{\e\in\Sp^{n-1}}z_1(x,\e)$ and $q_2(x)=\inf_{\e\in\Sp^{n-1}}z_2(x,\e)$, then $\lim_{r\rightarrow\infty}\big(q_i(r\xi)-r\big)=f(\xi)$ uniformly for $\xi\in\Sp^{n-1}$ and $i=1,2$.
Hence, there is a continuous function $s(r)>0$ with $s(r)\rightarrow0$ as $r\rightarrow\infty$ such that
\begin{equation}\aligned\label{q1q2}
q_1(r\xi)-s(r)\le r+f(\xi)\le q_2(r\xi)+s(r) \qquad\ \ \mathrm{for}\ \mathrm{any}\ \xi\in\Sp^{n-1}.
\endaligned
\end{equation}

We extend $f$ to $\tilde{f}$ by $\tilde{f}(x)=f(\f x{|x|})$ for $|x|\ge1$. A simply computation shows
\begin{equation}\aligned
\p_{ij}\big(|x|+\tilde{f}(x)\big)=&\f{\de_{ij}}{|x|}-\f{x_ix_j}{|x|^3}+\sum_{k,l}\f{f_{kl}}{|x|^2}\left(\de_{ki}-\f{x_kx_i}{|x|^2}\right)
\left(\de_{lj}-\f{x_lx_j}{|x|^2}\right)\\
&-\f{x_if_j+x_jf_i}{|x|^3}-\f{\sum_kx_kf_k}{|x|^3}\left(\de_{ij}-3\f{x_ix_j}{|x|^2}\right).
\endaligned
\end{equation}
For any $\a=(\a_1,\cdots,\a_n)\in\Sp^{n-1}$ with $\lan\a,x\ran=0$ and sufficiently large $|x|$, we have $\sum_{i,j}\p_{ij}\big(|x|+\tilde{f}(x)\big)x_ix_j=0$ and $$\f1{2|x|}\le\sum_{i,j}\p_{ij}\big(|x|+\tilde{f}(x)\big)\a_i\a_j\le\f2{|x|}.$$ Hence, the matrix $\left(\p_{ij}\big(|x|+\tilde{f}(x)\big)\right)_{n\times n}$ could be diagonal to $\mathrm{diag}\{0,\mu_1,\cdots,\mu_{n-1}\}$ with $\f1{2|x|}\le\mu_i\le\f2{|x|}$ for $1\le i\le n-1$ and sufficiently large $|x|$.

For $\ep>0$, let $$f_\ep(x)=\int\r(y)\left(|x-\ep y|+\tilde{f}(x-\ep y)\right)dy.$$
Then $f_\ep(x)$ is convex for sufficiently large $|x|$ and $\widetilde{\G}_k\triangleq\{x\in\R^n\big|\ |x|+\tilde{f}(x)=k\}$ has the mean curvature $0\le H_{\widetilde{\G}_k}\le\f{2(n-1)}{|x|}$ for large $k$.

Denote $B_1$ be the ball with radius 1 and centered at the origin in $\R^n$.
Since $f_{\ep}(x)\rightarrow|x|+\tilde{f}(x)$ in $C^2(K)$ for any compact set $K\subset\R^n\setminus B_1$, then there is a sufficiently small $\ep_k>0$ such that the mean curvature of the boundary of the smooth convex domain $\Om_k\triangleq\{x\in\R^n|\ f_{\ep_k}(x)<k\}$ satisfies $|H_{\p\Om_k}|\le1$ for each sufficiently large $k$. And we can assume $\lim_{k\rightarrow\infty}\ep_k=0$.

By Theorem \ref{Dirichlet}, there is a smooth solution $u_k$ to \eqref{ts} in $\Om_k$ with $u_k(x)=k$ on $\p \Om_k$ and $|Du_k|<1$ in $\Om_k$. Due to $|H_{\p\Om_k}|\le1$ and Lemma \ref{Conv}, $u_k$ is convex. Similar to \eqref{Vtep}, one gets $$\big|u_k(x)-|x|-\tilde{f}(x)\big|=\big|f_{\ep_k}(x)-|x|-\tilde{f}(x)\big|\le\ep_k\ \ \ \mathrm{for}\ x\in\p\Om_k.$$
Then for any $r\xi\in\p\Om_k$, combining \eqref{q1q2} gives
$$q_1(r\xi)-s(r)-\ep_k\le r+f(\xi)-\ep_k\le u_k(r\xi)\le r+f(\xi)+\ep_k\le q_2(r\xi)+s(r)+\ep_k.$$
Hence by Lemma \ref{comp} and the definitions of $q_1,q_2$ we conclude $$q_1(x)-\sup_{x\in\p\Om_k}s(|x|)-\ep_k\le u_k(x)\le q_2(x)+\sup_{x\in\p\Om_k}s(|x|)+\ep_k$$ for any $x\in \Om_{k}$ and sufficiently large $k$. By Lemma \ref{Converge}, we get an entire smooth convex strictly spacelike function $u$ satisfying \eqref{tsl} with $q_1(x)\le u(x)\le q_2(x)$, which implies
\begin{equation}\aligned\label{uxf}
\lim_{|x|\rightarrow\infty}\left(u(x)-|x|-f\big(\f x{|x|}\big)\right)=0.
\endaligned
\end{equation}
By Lemma \ref{max}, we know the uniqueness of the solution to \eqref{tsl} satisfying \eqref{uxf}.
\end{proof}

\section{An application to the conjecture of Aarons}

Up to a scaling we can assume $a=1$ in the equation \eqref{tsf}, namely,
\begin{equation}\aligned\label{tsc}
\div\left(\f{Du}{\sqrt{1-|Du|^2}}\right)=\f 1{\sqrt{1-|Du|^2}}+c.
\endaligned
\end{equation}
In this section, we only consider the case $c>0$. Using the notations in $\S$ 2, \eqref{tsc} is equivalent to
\begin{equation}\aligned\label{HnuEc}
H=c-\lan\nu,E_{n+1}\ran,
\endaligned
\end{equation}
and $\mathcal{L}u=1+c\sqrt{1-|Du|^2}$. Moreover, \eqref{tsc} in $\Om$ is the Euler-Lagrange equation of the variational problem
$\sup_{z\in \mathcal{C}(\varphi,\Om)}F_{c,\Om}(z)$. Sometimes, we denote $F_{c,\Om}$ by $F_c$ for simplicity if there is no ambiguity in the text. Let $v=\sqrt{1-|Du|^2}$. Compared to Theorem \ref{uni}, one has
\begin{theorem}\label{unic}
Let $\Om$ be a bounded domain in $\R^n$ and $u$ be a smooth strictly spacelike solution to \eqref{tsc} in $\overline{\Om}$. Set $M=\{(x,u(x))\big|\ x\in\Om\}$ and $\Si=\{(x,w(x))\big|\ x\in\Om\}$ with $w\in \mathcal{C}(u,\Om)$, then
\begin{equation}\aligned
\int_\Om e^{-w}\left(\sqrt{1-|Dw|^2}+c\right)dx\le\int_\Om e^{-u}\left(\sqrt{1-|Du|^2}+c\right)dx,
\endaligned
\end{equation}
where the above inequality attains equality if and only if $\Si=M$.
\end{theorem}
\begin{proof}
We only need to make a few changes in the proof of Theorem \ref{uni}. Define the vector field $Y$ by  $\sum_{i=1}^n\f{u_i}{v}e^{-x_{n+1}}E_i+\left(\f{1}{v}+c\right)e^{-x_{n+1}}E_{n+1},$ then
\begin{equation}\aligned
\overline{\div}(Y)=\sum_{i=1}^n\p_{x_i}\left(\f{u_i}{v}\right)e^{-x_{n+1}}-\left(\f{1}{v}+c\right)e^{-x_{n+1}}.
\endaligned
\end{equation}
If $\xi=(\xi_1,\cdots,\xi_{n+1})$ is a timelike future-pointing vector in $\R_1^{n+1}$, then
\begin{equation}\aligned
-\lan Y,\xi\ran=&-\sum_{i=1}^n\f{u_i\xi_i}ve^{-x_{n+1}}+\left(\f{1}{v}+c\right)\xi_{n+1}e^{-x_{n+1}}\\
=&c\xi_{n+1}e^{-x_{n+1}}+\left(\f{\xi_{n+1}}v-\sum_{i=1}^n\f{u_i\xi_i}v\right)e^{-x_{n+1}}\\
\ge& c\xi_{n+1}e^{-x_{n+1}}+\sqrt{\xi_{n+1}^2-\sum_{i=1}^n\xi_i^2}e^{-x_{n+1}}.
\endaligned
\end{equation}
By following the steps of the proof for Theorem \ref{uni}, we complete the proof.
\end{proof}
With Theorem \ref{unic} and the proof of Lemma \ref{comp}, we obtain the following result.
\begin{lemma}\label{compc}
Let $u$ be a strictly spacelike function to \eqref{tsc} in $\Om$ with $u\big|_{\p\Om}=\varphi$. Let $\varphi_1,\varphi_2$ be weakly spacelike functions in $\overline{\Om}$ and $W_1\le c$, $W_2\ge c$ are continuous functions in $\Om$. If $w_i\in \mathcal{C}(\varphi_i,\Om)$ and $F_{W_i,\Om}(w_i)=\sup_{w\in \mathcal{C}(\varphi_i,\Om)}F_{W_i,\Om}(w)$ for $i=1,2$, then
\begin{equation}\aligned\label{comp2c}
&u(x)\le w_1(x)+\sup_{\p\Om}\big(\varphi-\varphi_1\big)\qquad \mathrm{for}\ x\in\Om\\
&u(x)\ge w_2(x)+\inf_{\p\Om}\big(\varphi-\varphi_2\big)\qquad \mathrm{for}\ x\in\Om.
\endaligned
\end{equation}
\end{lemma}
Theorem \ref{test} and its proof tell us that for any constant $C\in(-r,r)$ and $r>0$, the following ODE:
\begin{equation}\label{ODEc}\left\{\begin{split}
&\f{\phi''}{1-(\phi')^2}+\f{n-1}t\phi'=1+c\sqrt{1-(\phi')^2}\qquad \ \mathrm{for}\ t\in (0,r),\\
&\phi(r)=C,\ \ \ \phi(0)=0\ \ \ \mathrm{and}\ \  \ |\phi'|<1,\\
\end{split}\right.
\end{equation}
has a unique smooth solution $\phi_0$ in $(0,r)$. Furthermore, if $(1+c)r^2\le(n-1)C$ and $K_1>0$ is a constant with $w_{K_1}(r)=C$, then using Lemma \ref{max} gives
$$\f Crt\le\phi_0(t)\le w_{K_1}(t)\quad \mathrm{for}\ t\in[0,r].$$
If $C<0$, $r<\f1{1+c}$, $K_2<0$ is a constant with $\widetilde{w}_{K_2}(r)=C$ and $K_2^2\ge\f{(1+c)^2r^{2n+2}}{1-(1+c)^2r^2}$, then using Lemma \ref{max} gives $$\widetilde{w}_{K_2}(t)\le\phi_0(t)\le \f{C}rt\quad \mathrm{for}\ t\in[0,r].$$
Hence we have Lemma \ref{line} for the functional $F_c$, which could enable us to arrive at the following Lemma similar to Lemma \ref{Converge}.
\begin{lemma}\label{Convergec}
Let $\Om_k$ be a family of convex domains with smooth boundaries ($\p\Om_k$ may be empty), $\overline{\Om}_k\subset\Om_{k+1}$ and $\bigcup_{k\ge1}\Om_k=\R^n$. Let $u_k$ be a smooth convex strictly spacelike function to \eqref{tsc} in $\Om_k$. Denote $x_t=x+t(y-x)$ for $x,y\in\R^n$. Suppose that there is a function $W$ in $\R^n$ satisfying
$$\limsup_{t\rightarrow+\infty}\f{|W(x_t)-W(x_{-t})|}{|x_t-x_{-t}|}<1$$
for any $x\neq y\in\R^n$. If $\limsup_{k\rightarrow\infty}|u_k(x)-W(x)|\le C$ for some absolute constant $C$ and any $x\in\R^n$, then there is a subsequence of $\{u_k\}$ converging to an entire smooth convex strictly spacelike function $u$ to \eqref{tsc} in $\R^n$.
\end{lemma}

Since $g^{ij}u_{ij}=1+c\sqrt{1-|Du|^2}$, then compared with \eqref{bge} one has
\begin{equation}\aligned\label{bgec}
g^{ij}(|Du|^2)_{ij}+\f1{v^2}g^{ij}(|Du|^2)_i(|Du|^2)_j+\f cvu_i(|Du|^2)_i=2g^{ij}u_{ki}u_{kj}\ge0.\\
\endaligned
\end{equation}
Let $L$ be a second order differential operator as previously defined by
$$Lf=e^{u}\div_M(e^{-u}\na f)\qquad \mathrm{for}\ f\in C^2(\R^n).$$
Let $B$ be the second fundamental form and $B_{e_ie_j}=h_{ij}\nu$ as $\S$ 4. After few modifications in the proof of Lemma \ref{LhijB}, we have
\begin{equation}\aligned\label{Lhc}
Lh_{ij}=|B|^2h_{ij}-c\sum_kh_{ik}h_{jk}.
\endaligned
\end{equation}
If we consider the level set of $u$ satisfying \eqref{tsc}, then comparing to \eqref{ug}\eqref{ugg} gives
\begin{equation}\aligned\label{ugc}
H_\G u_\g+\f{u_{\g\g}}{1-u_{\g}^2}=1+c\sqrt{1-u_\g^2},
\endaligned
\end{equation}
which implies
\begin{equation}\aligned\label{uggc}
u_{\g\g}=\left(1-u_{\g}^2\right)\left(1+c\sqrt{1-u_\g^2}-H_\G u_\g\right).
\endaligned
\end{equation}
In view of the proof of Lemma \ref{Conv}, we have
\begin{lemma}\label{Convc}
If $u$ is a smooth strictly spacelike solution to \eqref{tsc} in $\overline{\Om}_{h,u}$ and $\G_{h,u}$ is a nonempty convex compact set in $\R^n$ with mean curvature $H_\G\le1$, then $u$ is convex in $\Om_{h,u}$.
\end{lemma}

By \cite{JLJ}, the elliptic equation \eqref{tsc} has a convex radially symmetric strictly spacelike solution $\psi(r)$, where $\psi(r)\in C^\infty([0,+\infty))$ satisfies the following ODE
\begin{equation}\aligned\label{rtsc}
\f{\psi''}{1-(\psi')^2}+\f{n-1}r\psi'=1+c\sqrt{1-(\psi')^2}\qquad \mathrm{for}\ r\in (0,+\infty)
\endaligned
\end{equation}
with $\psi'(0)=0$ in an appropriate coordinate system. Owing to the estimate of $\psi'$ in \cite{JLJ}, $\psi(r)-\psi(0)$ is uniformly bounded. Particularly, the ODE
$$\psi''=\left(1-(\psi')^2\right)\left(1+c\sqrt{1-(\psi')^2}\right)\qquad \mathrm{for}\ r\in (0,+\infty)$$
has a smooth strictly spacelike solution. Let $\Om_\si$ be a domain defined in Lemma \ref{tds}, if $u_\si$ is a smooth strictly spacelike solution to \eqref{tsc} in $\Om_\si$ with $u_\si\big|_{\p\Om_\si}=0$, then the gradient of $u_\si$ is uniformly bounded. Let
\begin{equation*}\aligned
\widetilde{Q}_\sigma f=\sum_{i,j}\left(\de_{ij}+\f{f_if_j}{1-|Df|^2}\right)f_{ij}-\sigma\big(1+c\sqrt{1-|Df|^2}\big).
\endaligned
\end{equation*}
If the maximal eigenvalue of the matrix $(h_{ij})$ is a sufficiently large positive number in any bounded domain $\Om$, then it must attain this maximal eigenvalue on the boundary $\p\Om$ by the equation \eqref{Lhc}.
Combining \eqref{uggc} and Lemma \ref{Convc} we follow the steps of the proof for Theorem \ref{Dirichlet} with corresponding modifications and obtain the following Theorem.
\begin{theorem}\label{Dirichletc}
Let $\Om$ be a bounded convex domain with smooth boundary in $\R^n$. Then the Dirichlet problem $\mathcal{L}u=1+c\sqrt{1-|Du|^2}$ in $\Om$ with $u\big|_{\p\Om}=0$ has a smooth strictly spacelike solution $u$ in $\overline{\Om}$.
\end{theorem}
In conjunction with Lemma \ref{compc}, Lemma \ref{Convergec}, Lemma \ref{Convc}, \eqref{rtsc} and Theorem \ref{Dirichletc}, now we could construct many solutions to \eqref{tsc} along the proof of Theorem \ref{Vcu}.
\begin{theorem}
For $n\ge2$ and any $V\in Q\setminus Q_0$ there is an entire smooth convex strictly spacelike solution $u$ to \eqref{tsc} with $c>0$ such that $u$ blows down to $V$.
\end{theorem}

\bigskip

\section{Appendix   Semicontinuity for concave functionals}

Let $\{u_k\}$ be a sequence of weakly spacelike functions which converges to $u_0\in \mathcal{C}(\varphi,\Om)$ uniformly. Without loss of generality, we assume $u_k\rightharpoonup u_0$ weakly in Sobolev space $H^1(\Om)$. Since locally Lipschitz functions are differentiable almost everywhere, then by the theorems of Egorov and Lusin, for any $\de>0$ there is an open set $\widetilde{\Om}\subset\Om$ such that the measure of $\Om\setminus\widetilde{\Om}$ is less than $\de$ and $u_k,u_0\in C^1(\widetilde{\Om})$ for all $k\ge1$.

For $0<\ep<1$, the concavity of $\sqrt{1-|p|^2}$ gives
\begin{equation}\aligned\label{concave}
\sqrt{1-\ep^2|Du_k|^2}-\sqrt{1-\ep^2|Du_0|^2}\le-\f{\ep^2Du_0\cdot D(u_k-u_0)}{\sqrt{1-\ep^2|Du_0|^2}}\qquad \ \ \mathrm{in}\ \ \widetilde{\Om}.
\endaligned
\end{equation}
In fact, the inequality \eqref{concave} is equivalent to
\begin{equation}\aligned\label{rch}
\sqrt{1-\ep^2|Du_0|^2}\sqrt{1-\ep^2|Du_k|^2}\le1-\ep^2Du_0\cdot Du_k.
\endaligned
\end{equation}
The inequality \eqref{rch} is just the reversed Cauchy-Schwarz inequality (see\cite{O} for example).
Then
\begin{equation}\aligned
&\int_{\widetilde{\Om}} e^{-u_k}\sqrt{1-\ep^2|Du_k|^2}dx\\
=&\int_{\widetilde{\Om}} e^{-u_0}\sqrt{1-\ep^2|Du_k|^2}dx+\int_{\widetilde{\Om}} \big(e^{-u_k}-e^{-u_0}\big)\sqrt{1-\ep^2|Du_k|^2}dx\\
\le&\int_{\widetilde{\Om}} e^{-u_0}\sqrt{1-\ep^2|Du_0|^2}dx-\int_{\widetilde{\Om}} e^{-u_0}\f{\ep^2Du_0\cdot D(u_k-u_0)}{\sqrt{1-\ep^2|Du_0|^2}}dx\\
&+\int_{\widetilde{\Om}} \big(e^{-u_k}-e^{-u_0}\big)\sqrt{1-\ep^2|Du_k|^2}dx.
\endaligned
\end{equation}
Hence, we obtain
\begin{equation}\aligned
\limsup_{k\rightarrow\infty}\int_{\widetilde{\Om}} e^{-u_k}\sqrt{1-|Du_k|^2}dx\le&\limsup_{k\rightarrow\infty}\int_{\widetilde{\Om}} e^{-u_k}\sqrt{1-\ep^2|Du_k|^2}dx\\
\le&\int_{\widetilde{\Om}} e^{-u_0}\sqrt{1-\ep^2|Du_0|^2}dx.
\endaligned
\end{equation}
Let $\ep\rightarrow1^-$, then
\begin{equation}\aligned\label{A1}
\limsup_{k\rightarrow\infty}\int_{\widetilde{\Om}} e^{-u_k}\sqrt{1-|Du_k|^2}dx
\le\int_{\widetilde{\Om}} e^{-u_0}\sqrt{1-|Du_0|^2}dx.
\endaligned
\end{equation}
By the choice of $\widetilde{\Om}$ and \eqref{A1}, we conclude that
\begin{equation}\aligned\label{A3}
\limsup_{k\rightarrow\infty}\int_{\Om} e^{-u_k}\sqrt{1-|Du_k|^2}dx
\le\int_{\Om} e^{-u_0}\sqrt{1-|Du_0|^2}dx.
\endaligned
\end{equation}

\bigskip

\bibliographystyle{amsplain}

\end{document}